\documentclass[sn-mathphys,Numbered]{sn-jnl}

\usepackage{graphicx}%
\usepackage{multirow}%
\usepackage{amsmath,amssymb,amsfonts}%
\usepackage{amsthm}%
\usepackage{mathrsfs}%
\usepackage[title]{appendix}%
\usepackage{xcolor}%
\usepackage{textcomp}%
\usepackage{manyfoot}%
\usepackage{booktabs}%
\usepackage{algorithm}%
\usepackage{algorithmicx}%
\usepackage{algpseudocode}%
\usepackage{listings}%

\theoremstyle{thmstyleone}%
\newtheorem{theorem}{Theorem}
\theoremstyle{thmstyletwo}%

\theoremstyle{thmstylethree}%
\newtheorem{definition}{Definition}%

\begin{document}

\title[A Randomized Block Krylov Method for Tensor Train Approximation]{A Randomized Block Krylov Method for Tensor Train Approximation}


\author*[1]{\fnm{Gaohang} \sur{Yu}}\email{maghyu@163.com}

\author[1]{\fnm{Jinhong} \sur{Feng}}\email{fengjinhong0502@163.com}

\author[1]{\fnm{Zhongming} \sur{Chen}}\email{zmchen@hdu.edu.cn}
\author[2]{\fnm{Xiaohao} \sur{Cai}}\email{x.cai@soton.ac.uk}
\author[1,3]{\fnm{Liqun} \sur{Qi}}\email{liqun.qi@polyu.edu.hk}

\affil[1]{\orgdiv{Department of Mathematics}, \orgname{Hangzhou Dianzi University}, 
 \country{China}}

\affil[2]{\orgdiv{School of Electronics and Computer Science}, \orgname{University of Southampton}, 
\city{Southampton}, 
 \country{UK}}

\affil[3]{\orgdiv{Huawei Theory Research Lab}, 
\city{Hong Kong}, 
\country{China}}


\abstract{Tensor train decomposition is a powerful tool to tackle high-dimensional large-scale tensor data and is not suffering from the curse of dimensionality. It is based on the low-rank approximation of auxiliary unfolding matrices. To accelerate the calculation of the auxiliary unfolding matrices, some randomized algorithms have been proposed; however, they are not suitable for noisy data. The randomized block Krylov method is capable of dealing with heavy-tailed noisy data in low-rank approximation of matrices. In this paper, we propose a randomized algorithm for low-rank tensor train approximation of large-scale tensors based on randomized block Krylov subspace iteration and provide theoretical guarantees. Extensive numerical experiments on synthetic and real-world tensor data demonstrate the great performance of the proposed algorithm.	}

\keywords{Tensor, tensor train decomposition, randomized algorithm,  randomized block Krylov subspace iteration}

\pacs[MSC Classification]{68W20, 15A18, 15A69}

\maketitle

\section{Introduction}\label{sec1}
In many practical applications, such as machine learning, compressed sensing and data mining, data analysis generally plays the key role. Many large-scale datasets can be naturally expressed as tensors, often exhibiting low-rank structures. As a result, the low-rank tensor approximation has become a powerful tool for tensor data analysis \cite{ASLTucker,ASL2020,comon2002,kolda,czm2018,Zhao21}.

The main models for low-rank tensor approximation include the  CANDECOMP/PARAFAC (CP) decomposition \cite{comon2002}, Tucker decomposition \cite{kolda}, T-product\cite{kilmer2011}, tensor train (TT) decomposition \cite{Oseledets2011}, and tensor ring (TR) decomposition \cite{TR2016}, etc. CP decomposition is a useful way to represent a large tensor as the sum of a series of rank-1 tensors. Unfortunately, it is not reliable due to the difficulty of finding the optimal CP rank. Tucker decomposition is more stable than CP decomposition, but it is affected by the curse of dimensionality. In comparison, TT decomposition is not affected by the curse of dimensionality and is more reliable. This paper will focus exclusively on TT decomposition, which is becoming increasingly popular due to its stability and efficiency.

TT decomposition is based on the low-rank approximation of auxiliary unfolding matrices. As the computation of the singular value decomposition (SVD) of large-scale auxiliary unfolding matrices is a time-consuming task, various solutions have been proposed in literature. One of these solutions is the randomized algorithms \cite{B2017,Gu2014,Halko2011,Tropp2017,Tropp2019,Wang2015}, which have been shown to be effective in computing low-rank approximations of large-scale matrices. Randomized algorithms can significantly speed up decomposition and produce highly accurate results.

In recent years, numerous researchers have proposed randomized tensor decomposition methods \cite{STTA,Che2021,Che2020,DYQC2023,ASLTucker,QiL and Yu,SunY2020,Q2022,wolf2019,YuLi2022,YuLi2023,Jzhang2018,Lyuan2019,Zhao22}, showing the well-established effectiveness of randomized algorithms in various tensor decompositions.  Huber et al. \cite{Huber2017} proposed randomized TT decomposition based on a simple randomized SVD (TT-rSVD), which is much faster than the classical deterministic TT algorithm (i.e., TT-SVD) \cite{Oseledets2011}. To address unknown exact TT rank situations, Che et al. \cite{Che2018} proposed an adaptive randomized TT decomposition. To take advantage of parallelization, Shi et al. \cite{s2021} proposed parallelizable TT decompositions, including PSTT2 and PSTT2-onepass (which requires only one pass through the raw data). Daas et al.  proposed parallel algorithms for TT arithmetic \cite{Daas2020} and proposed randomized algorithms for rounding in the TT format \cite{Daas2023}. Kressner et al.  \cite{STTA} proposed streaming TT approximation using two-sided sketching to make high-dimensional data streamable and easy to implement in parallel. Li et al.  \cite{Li2022} proposed a fast TT for sparse tensors.


The randomized algorithm projects the original matrix onto a subspace using a random matrix, and power iteration can improve the accuracy of the approximation by reducing the tail energy \cite{Wang2015,Tropp2018,Gu2014}. However, the simultaneous power method only retains the highest-order term in the Krylov subspace, which limits its applicability and prevents it from handling the problem of under-sized singular value gaps. To address this issue, Musco et al. \cite{Musco2015} proposed the randomized block Krylov subspace iteration method (rBKI), which performs better in experiments and does not depend on singular value gaps. Moreover, the approximation can be as accurate as the truncated SVD for rapidly decaying singular values. Recently, Qiu et al. \cite{Q2022}  proposed a randomized Tucker decomposition based on block Krylov iteration and proved its validity on noisy data.

In TT decomposition, direct use of the random algorithm to approximate the auxiliary unfolding matrices will lead to excessive tail energy. The accumulated error will be larger than that of TT-SVD.
Given the advantage of randomized methods in reducing tail energy, in this paper, we propose a method based on randomized block Krylov methods for TT decomposition. Combining randomized block Krylov strategies, our proposed method can achieve more accurate TT approximation than existing standard randomized algorithms. Moreover, our method inherits the ability of the randomized block Krylov subspace iteration in matrices to handle noisy data and can achieve near-optimal approximated relative errors.
Extensive experiments in synthetic and real-world data demonstrate the excellent performance of the proposed randomized method for low-rank TT approximation.

The rest of the paper is organized as follows. Section \ref{Sec:pre} briefly introduces symbols, definitions, tensor basics, and recalls the related work. In Section \ref{Sec:our-method}, we present our main method and analyze its probabilistic error bound. Extensive numerical results and comparisons validating the great performance of our proposed method are given in Section \ref{Sec:results}. We conclude in Section \ref{Sec:con}.

\begin{table}[htbp]
	\centering
	\caption{Main notations.}
	\begin{tabular}{c|c}
		\hline
		Symbols & Notation\\
		\hline
		$x$  & scalar\\
		$X $ &  matrix\\
		$\mathcal{X}$ & tensor\\
		$\mathcal{X}^{(n)} $ & mode-$n$ unfolding matrix of $\mathcal{X}$\\
		$\times_n$ & mode-$n$ product of tensor and matrix \\
		$\times_n^m$ & mode-$(n,m)$ product of tensor and tensor \\
		$X^\top$ & transpose of $X$\\
		$X^{\dagger}$ & pseudo-inverse of $X$\\
		$ \Xi$  & expectation \\
		\hline
	\end{tabular}%
	\label{tab:notation}%
\end{table}%

\section{{Preliminary}} \label{Sec:pre}
\subsection{{Notations and basic operations}}
The main symbols used in this paper are shown in Table \ref{tab:notation}.
The mode-$n$ product of tensor $\mathcal{A}\in \mathbb{R}^{I_1\times I_2\times \cdots\times I_N} $ by a matrix $B \in \mathbb{R}^{J_n \times I_n}$ is a tensor $\mathcal{C}\in \mathbb{R}^{I_1\times \cdots \times I_{n-1}\times J_n \times I_{n+1} \times \cdots\times I_N} $, denoted by $\mathcal{C} = \mathcal{A} \times_n B $.
The tensor-tensor product of two tensors  $\mathcal{A}\in \mathbb{R}^{I_1\times I_2\times \cdots\times I_N} $  and $\mathcal{B}\in \mathbb{R}^{J_1\times J_2\times \cdots\times {J_M}}$ with common modes $I_n = J_m$ produces an $(M+N-2)$-th order tensor $\mathcal{C}\in \mathbb{R}^{I_1\times \cdots \times I_{n-1}\times  I_{n+1} \times \cdots \times I_N \times  J_1\times \cdots \times J_{m-1}\times  J_{m+1} \times \cdots \times J_M}$, i.e.,
\begin{equation*}
	\mathcal{C} = \mathcal{A} \times_n^m \mathcal{B}.
\end{equation*}
The Frobenius norm of a tensor$\mathcal{A}$ is given by $\left \| \mathcal{A} \right \|_F  = \sqrt{\left \langle \mathcal{A}  ,\mathcal{A}   \right \rangle}$.

\begin{definition}{(Matricization \cite{Huber2017})}
	The $\alpha$-Matricization is defined as \begin{equation} \hat{M}_{\alpha}: \mathbb{R}^{I_1 \times I_2\times \cdots \times I_N  } \to \mathbb{R}^{m_{\alpha} \times m_{\beta}},
    \end{equation}
 where $ m_{\alpha} = I_1 \times I_2 \times \cdots \times I_{i-1}$ and $m_{\beta}  = I_i \times I_{i+1} \times \cdots \times I_N$.
\end{definition}

\begin{definition}{(TT decomposition \cite{Oseledets2011})}
	Given a tensor $\mathcal{A}\in \mathbb{R}^{I_1 \times I_2\times \cdots\times I_N}$, the TT decomposition of $\mathcal{A}$ with ${\rm rank}_{\rm TT}=\textbf{r}=(r_1,r_2,\cdots,r_{N-1})$ is expressed as
	\begin{equation}
		\mathcal{A} \approx \mathcal{Q}_{1} \times_3^1 \mathcal{Q}_{2}\times_3^1\cdots\times_3^1\mathcal{Q}_{N},
	\end{equation}
	where $ \mathcal{Q}_n \in \mathbb{R}^{r_{n-1} \times I_n \times r_n}$ is a factor tensor for $n=1,2,\ldots,N$ and $r_0 = r_N =1 $.
\end{definition}

\begin{definition}{(Tail energy)} The $j$-th tail energy of a matrix $X$ is defined as	
	\begin{equation}
		\tau_j^2(X)  = \min_{{\rm rank}(Y) <j } \left \| X-Y \right \|_F^2= \sum_{i \geq j } \sigma_i^2(X),
	\end{equation}
	where each $\sigma_i$ is a singular value.	
\end{definition}

\subsection{TT-SVD}
The classical TT decomposition (i.e., TT-SVD) was proposed by Osledets \cite{Oseledets2011}. It was implemented based on the truncated SVD of the auxiliary unfolding matrix using precision to control the truncation parameters. The TT-SVD algorithm is given in Algorithm \ref{alg:tt-svd}. 

\begin{algorithm}
	\caption{TT-SVD \cite{Oseledets2011}}
	\label{alg:tt-svd}
	\begin{algorithmic}[1]
		\Require Tensor $\mathcal{A} \in\mathbb{R}^{I_1\times I_2\times\cdots\times I_N}$ and the prescribed accuracy $\epsilon$.
		\Ensure  Cores
		$ \mathcal{Q}_{1},\mathcal{Q}_{2},\cdots,\mathcal{Q}_{N}$
		of the TT approximation $\mathcal{B}$ to $\mathcal{A}$ in TT-format  satisfying
		$\left\| \mathcal{A}-\mathcal{B} \right \|_F \leq \epsilon \left\| \mathcal{A} \right \|_F $.
		\State Initialization: truncation parameter $\delta =\frac{\epsilon}{\sqrt{N-1}} \left\| \mathcal{A} \right \|_F$, $C=\mathcal{A}^{(1)}, r_0=1$.
		
		\For{$n=1,2,\cdots,N-1$}
		\State $C ={\tt reshape}(C,[r_{n-1}I_n,I_{n+1} I_{n+2}\cdots  I_{N}])$.
		\State Compute the $\delta$-truncated SVD: $C = USV+E, \left\| E\right\|_F \leq \delta , r_n = {\rm rank}(S)$.
		
		\State Form the factor tensor $ \mathcal{Q}_{n} = {\tt reshape}(U,[r_{n-1},I_n,r_n])$.
		\State Update $ C  = SV^\top $.
		\EndFor         
		
		\State Form the last factor tensor with $r_N=1 $, $ \mathcal{Q}_{N} = {\tt reshape}(C,[r_{N-1},I_N,r_N])$.
		\State \textbf{return} $\mathcal{Q}_{1},\mathcal{Q}_{2},\cdots,\mathcal{Q}_{N}$.
	\end{algorithmic}
\end{algorithm}

Let $I = \max\{I_1, \ldots, I_{N}\}$ and $r= \max\{r_1, \ldots, r_{N-1} \}$. The computational complexity of TT-SVD is dominated by the $(N-1)$ matrix SVD, i.e., $\mathcal{O}(I^{N+1} +r^2 \sum_{i=1}^{N-3}I^{N+1-i} + r I^3)$,
which is still of exponential order. Unfortunately, tensors with e.g. sparse structure do not bring significant improvement in efficiency because the structure information is usually lost after the first SVD. As shown in the following Theorem \ref{thmSketch}, errors in TT decomposition are accumulated.

\begin{theorem} [Theorem 4.3 \cite{Oseledets2011}] \label{thmSketch}
	Suppose that the unfoldings $A_k$ of the tensor $ \mathcal{A}$ satisfy
	\begin{equation}
		A_k = R_k + E_k, \ {\rm rank}(R_k) = r_k, \ \left\| E_k \right \|_F =\epsilon_k, \ k = 1,2,\cdots,N-1.
	\end{equation}
	Then TT-SVD computes a tensor $ \mathcal{B}$  with ${\rm rank}_{\rm TT}=(r_1,r_2,\cdots,r_{N-1})$ in the TT-format and
	\begin{equation}
		\left\| \mathcal{A}-\mathcal{B} \right \|_F \leq \sqrt{\sum_{k=1}^{N-1}\epsilon_k^2}.
	\end{equation}
\end{theorem}

\subsection{Randomized TT-SVD}
Although TT decomposition can alleviate the curse of dimensionality, the cost of calculating the auxiliary matrix is still of exponential order, mainly due to the SVD of large-scale matrices. Randomized algorithms can significantly accelerate this process. Huber et al. \cite{Huber2017} proposed a randomized algorithm for TT decomposition, i.e., TT-rSVD, see Algorithm \ref{alg:tt-rsvd}. The matlab function ${\tt qr}(Y,0)$ produces the ``economy size" QR decomposition.

\begin{algorithm}
	\renewcommand{\algorithmicrequire}{\textbf{Input:}}
	\renewcommand{\algorithmicensure}{\textbf{Output:}}
	\caption{ Randomized TT decomposition (TT-rSVD)\cite{Huber2017}}
	\label{alg:tt-rsvd}
	\begin{algorithmic}[1]
		\Require Tensor $\mathcal{A} \in\mathbb{R}^{I_1\times I_2\times\cdots\times I_N}$, target rank $\mathbf{r}= [r_1,r_2,\cdots ,r_{N-1}]$, oversampling parameter $p$, and $r_0 =1$.
		\Ensure $\mathcal{Q}_{1},\mathcal{Q}_{2},\cdots,\mathcal{Q}_{N}$.

		\For{$n=1,2,\cdots,N-1$}
		\State $A = {\tt reshape}(\mathcal{A},[r_{n-1}I_n,I_{n+1} I_{n+2}\cdots  I_{N}])$.
		\State Create random Gaussian matrix $\Omega \in R^{I_{n+1} I_{n+2}\cdots  I_{N} \times (r_n+p)} $.
		\State Calculate $Y = A \Omega$.
		\State Calculate $[Q,\sim] = {\tt qr}(Y,0)$.
		
		\State $Q = Q(:,1:r_n)$.
		\State Form the factor tensor $ \mathcal{Q}_{n} = {\tt reshape}(Q,[r_{n-1},I_n,r_n])$.
		\State Update $ \mathcal{A}  = \mathcal{A} \times_1 Q^\top \in R^{r_n \times I_{n+1}\times\cdots\times I_N} $.
		\EndFor          
		\State Form the last factor tensor with $r_N=1 $, $ \mathcal{Q}_{N} = {\tt reshape}(\mathcal{A},[r_{N-1},I_N,r_N])$.
		\State \textbf{return} $\mathcal{Q}_{1},\mathcal{Q}_{2},\cdots,\mathcal{Q}_{N}$.
	\end{algorithmic}
\end{algorithm}

However, when using randomized methods to calculate TT decomposition, the size of the random matrix determines the generated error, and the error of each step will accumulate. The accumulation of errors can cause considerable tail energy and result in the loss of too much original information.
To address this problem, oversampling techniques can be employed in randomized algorithms. This technique involves using a larger random matrix to capture more information in the data matrix and then truncating it to improve the accuracy of randomized methods. The error bound of TT-rSVD is given in the following Theorem \ref{thm:tt-rsvd}.

\begin{theorem} [See \cite{Huber2017}] \label{thm:tt-rsvd}
	Given $\mathcal{A} \in R ^{I_1 \times I_2 \times \cdots \times I_N}$ and $s=r+p$ with $ p \geq 4$. For every $u,t \geq 1$, the error bound of TT-rSVD reads	
	\begin{equation}
		\left \| \mathcal{A} - P_{2,\cdots,N}(\mathcal{A}) \right\| \leq \sqrt{N-1} \eta(r,p) \min_{{\rm rank}_{\rm TT}(\mathcal{B}) \leq \textbf{r}  } { \left\| \mathcal{A}-\mathcal{B} \right \|},
	\end{equation}
	with probability at least $(1-5t^{-p}-2e^{-u^2/2})^{N-1}$, where $P_{2,\cdots,N}(\mathcal{A})= \mathcal{Q}_1 \times_3^1 \mathcal{Q}_2 \times_3^1 \cdots \times_3^1
	\mathcal{Q}_N$ and $\eta $ is given by
	\begin{equation}
		\eta  =  1+ t\sqrt{\frac{12r}{p}} +ut\frac{e \sqrt{s}}{p+1}.
	\end{equation}
\end{theorem}


The subspace power iteration technology \cite{Gu2014} is an effective method to enhance the accuracy of randomized methods. Replacing $A$ with $(AA^\top)^qA$ can make the singular values of the original matrix decay faster, resulting in smaller tail energy and improved accuracy of randomized methods as shown in the following theorem.

\begin{theorem}[Corollary 10.10 in \cite{Halko2011}] \label{th:rsi}
	Define $B = (AA^\top)^qA$. For a nonnegative integer q and the oversampling parameter p, generate a random matrix $\Omega $, construct the sample matrix $Z = B\Omega$, and calculate $[Q_Z,\sim] = {\tt qr}(Z,0)$. Let $P_Z = Q_ZQ_Z^\top$, and then
	\begin{equation}
		\Xi \left \|  (I-P_Z)A \right\| \leq
		\Big [ \big(1+ \sqrt{\frac{r}{p-1}} \big)\delta_{r+1} ^{2q+1}+\frac{e \sqrt{s}}{p}    \big ( \sum_{j>r} \delta_j ^{2(2q+1)} \big )^\frac{1}{2}    \Big]^\frac{1}{2q+1}.
	\end{equation}
	
\end{theorem}

Musco et al. \cite{Musco2015} proposed an improved version of the simultaneous power iteration, a simple randomized block Krylov method, which gives the same guarantee in just $\mathcal{O}(\frac{1}{\sqrt{\epsilon}})$ iterations (recall that $\epsilon$ is the prescribed accuracy) and performs substantially better experimentally. In the following section, we propose our new randomized method based on the block Krylov iteration for TT decomposition.

\section{Proposed Block Krylov Iteration for TT  Approximation} \label{Sec:our-method}

We emphasize again that one big disadvantage of TT-rSVD is its limited accuracy and inability in tackling noisy data. To address this issue, we propose a new method called TT-rBKI based on the block Krylov subspace iteration for TT approximation. 

Our TT-rBKI method is summarized in Algorithm \ref{alg:tt-rbki}. Comparing Step 6 of Algorithm \ref{alg:tt-rbki} with Step 4 of Algorithm \ref{alg:tt-rsvd} (i.e., TT-rSVD), we can see that TT-rSVD directly utilizes a Gaussian random matrix to sketch the unfolding matrix $A$, while our TT-rBKI takes advantage of the Krylov subspace to create a larger random projection matrix $U$ based on the data. This means that the TT-rBKI method produces a larger sketch $Y$ through a data-related random projection operator. It is evident that the sketch obtained by this way ensures that it contains more principal components of the data than the sketch obtained by the TT-rSVD method. This plays a key role in improving the accuracy and stability of the method for the TT approximation. 
The error bound of the proposed TT-rBKI method is provided in Theorem \ref{theorem:tt-rbki} below.

	\begin{algorithm}
		\renewcommand{\algorithmicrequire}{\textbf{Input:}}
		\renewcommand{\algorithmicensure}{\textbf{Output:}}
		\caption{Randomized TT approximation based on  block Krylov subspace iteration (TT-rBKI)}
		\label{alg:tt-rbki}
		\begin{algorithmic}[1]
			\Require Tensor $\mathcal{A} \in\mathbb{R}^{I_1\times I_2\times\cdots\times I_N}$, target rank $\textbf{r} = [r_1,r_2,\cdots ,r_{N-1}]$, power iteration parameter $q$, oversampling parameter $p$, and $r_0 =1$.
			\Ensure $\mathcal{Q}_{1},\mathcal{Q}_{2},\cdots,\mathcal{Q}_{N}$.			
			
			\For{$n=1,2,\cdots,N-1$}
			\State $A = {\tt reshape}(\mathcal{A},[r_{n-1}I_n,I_{n+1} I_{n+2}\cdots  I_{N}])$.
			\State Create random Gaussian matrix $\Omega \in R^{I_{n+1} I_{n+2}\cdots  I_{N} \times (r_n+p)} $.
			\State Construct Krylov space $K = [A^\top A\Omega,(A^\top A)^2 \Omega,\cdots,(A^\top A)^q \Omega] $.
			\State Calculate $[U,\sim] = {\tt qr}(K,0)$, where
			$U \in \mathbb{R}^{I_{n+1} I_{n+2}\cdots  I_{N} \times q(r_n+p)}$.
			\State Calculate $Y = AU \in \mathbb{R}^{r_{n-1}I_n \times q(r_n+p)} $.
			\State Set $(Q,\sim) = {\tt qr}(Y,0)$.
            \State $Q = Q(:,1:r_n).$
			
			\State Form the factor tensor $ \mathcal{Q}_{n} = {\tt reshape}(Q,[r_{n-1},I_n,r_n])$.
			\State Update $ \mathcal{A}  = \mathcal{A} \times_1 Q^\top \in R^{r_n \times I_{n+1}\times\cdots\times I_N} $.
			\EndFor 
			\State Form the last factor tensor with $r_N=1 $, $ \mathcal{Q}_{N} = {\tt reshape}(\mathcal{A},[r_{N-1},I_N,r_N])$.
			\State \textbf{return} $\mathcal{Q}_{1},\mathcal{Q}_{2},\cdots,\mathcal{Q}_{N}$.
		\end{algorithmic}
	\end{algorithm}

\begin{theorem}[Error bound of TT-rBKI] \label{theorem:tt-rbki}
		Given $\mathcal{A} \in R ^{I_1 \times I_2 \times \cdots \times I_N}, s=r+p$ with $ p \geq 4$, and nonnegative integer $q$. For every $u \geq 1$, the error bound of TT-rBKI is		
		\begin{equation}
			\left \| \mathcal{A} - P_{2,\cdots,N}(\mathcal{A}) \right\| \leq \sqrt{N-1} \eta(r,p,q) \min_{{\rm rank}_{\rm TT}(\mathcal{B}) \leq \textbf{r}  } { \left\| \mathcal{A}-\mathcal{B} \right \|},
		\end{equation}
		with probability at least $(1-5t^{-p}-2e^{-u^2/2})^{N-1}$,	where the parameter $\eta $ is given as
		\begin{equation}
			\eta  = \Big [ \big(1+ \sqrt{\frac{r}{p-1}} +\frac{e \sqrt{s}}{p} \big)     \Big]^\frac{1}{2q+1}.
		\end{equation}
\end{theorem}
	
\begin{proof}	
		For syntactical convenience, we define $B_i =(\hat{M}_{1,\cdots,{i-1}}(A) )$ and use orthogonality from right to left. Since
		$P_{2,\cdots,N}(X) $
		is an orthogonal projector, we have
		\begin{equation}
			\begin{aligned}
				\left\| \mathcal{A}-P_{2,\cdots,N}(\mathcal{A}) \right \|^2 &=\left\| \mathcal{A} \right \|^2-\left\| P_{2,\cdots,N}(\mathcal{A}) \right \|^2\\
				&=\left\| \mathcal{A} \right \|^2-\langle B_2,B_2 \rangle\\
				&=\left\| \mathcal{A} \right \|^2-\langle Q_2 B_3,Q_2 B_3\rangle\\
				&=\left\| \mathcal{A} \right \|^2-\langle B_3,Q_2^\top Q_2 B_3\rangle\\
				&=\left\| \mathcal{A} \right \|^2-\langle B_3,B_3-(I-Q_2^\top Q_2)B_3\rangle\\
				&=\left\| \mathcal{A} \right \|^2 -\left\| B_3 \right \|^2     + \langle B_3,(I-Q_2^\top Q_2)B_3\rangle\\
				&=\left\| \mathcal{A} \right \|^2 -\left\| B_3 \right \|^2     +\left\| (I-Q_2^\top Q_2)B_3 \right \|^2.
			\end{aligned}
		\end{equation}
		Then, iteratively, we can obtain	
		\begin{equation} \label{eq:12}
            \begin{aligned}
			\left\| \mathcal{A}-P_{2,\cdots,N}(\mathcal{A}) \right \|^2 & = \left\| \mathcal{A} \right \|^2-\left\| B_{N+1} \right \|^2 +\sum_{i=2}^{N}\left\| (I-Q_i^\top Q_i)B_{i+1} \right \|^2 \\
            & = \sum_{i=2}^{N}\left\| (I-Q_i^\top Q_i)B_{i+1} \right \|^2,
            \end{aligned}
		\end{equation}
		where the last step is obtained by using the fact that $B_{N+1}$ has the same norm as $\mathcal{A}$. Since $Q_i$ is obtained by the block Krylov iteration, which has the same error bound as the simultaneous iteration, by using Theorem \ref{th:rsi}, we have
		\begin{equation}
			\begin{aligned}
				& \ \left\| (I-Q_i^\top Q_i)B_{i+1} \right \|^2 \\
				\leq & \  \Bigg [ \Big [ \big(1 \!+\! \sqrt{\frac{r}{p-1}} \big) \delta_{j+1}^{2q+1}(B_{i+1}) \!+\! \frac{e \sqrt{s}}{p} \big ( \sum_{j>r} \delta_j ^{2(2q+1)}(B_{i+1}) \big )^\frac{1}{2}    \Big]^\frac{1}{2q+1} \Bigg] ^2\\
				\leq & \  \Bigg[ \Big [ \big(1 \!+\! \sqrt{\frac{r}{p-1}} +\frac{e \sqrt{s}}{p} \big)   \big ( \sum_{j>r} \delta_j ^{2(2q+1)}(B_{i+1}) \big )^\frac{1}{2}    \Big]^\frac{1}{2q+1} \Bigg]^2\\
				\leq & \  \eta ^2  \sum_{j>r} \delta_j ^2(B_{i+1}) .
			\end{aligned}
		\end{equation}
		
		As shown by Halko \cite{Halko2011}, the application of an orthogonal projection can only decrease the singular values. It follows that
		\begin{equation} \label{eq:14}
			\begin{aligned}
				\left\| (I-Q_i^\top Q_i)B_{i+1} \right \|^2  &\leq \eta ^2  \sum_{j>r} \delta_j ^2(B_{i+1}) \\
				& \leq \eta ^2 \min_{{\rm rank}_{\rm TT}(\mathcal{B} ) \leq r  } \left \| \mathcal{A}-\mathcal{B} \right \|^2.
			\end{aligned}
		\end{equation}
		By combining Eq. \eqref{eq:12} and Eq. \eqref{eq:14}, this completes the proof.
\end{proof}

For comparison purpose, we also provide a simple version of TT-rBKI with basic subspace power iteration, called TT-rSI, which is summarized in Algorithm \ref{alg:tt-rsi}. The TT-rSI method has a lower computational complexity than the TT-rBKI method (see Table \ref{tab2}), inasmuch as it only uses the information associated with the highest power in the Krylov subspace. On the other hand, this also makes it less capable of dealing with noisy data than the TT-rBKI method, as demonstrated by extensive numerical experiments in Section \ref{Sec:results}.

\begin{table}[h]
		\centering
		\caption{Computational complexity of methods TT-SVD, TT-rSVD, TT-rSI and TT-rBKI.} \label{tab2}
		\begin{tabular}[htbp]{cc}		
			\hline
			Method & Computational Complexity\\
			\hline
			TT-SVD&$\mathcal{O}(I^{N+1} +\sum_{i=1}^{N-1}r^2I^{N-i})$\\
			TT-rSVD&$\mathcal{O}(NsI^{N} )$\\
			TT-rSI&$\mathcal{O}(I^Nsq +\sum_{i=1}^{N-1}rI^{N-i}sq)$\\
			TT-rBKI&$\mathcal{O}(I^{N}s^2q^2 +\sum_{i=1}^{N-1}rI^{N-i}s^2q^2)$\\
			\hline
		\end{tabular}
	\end{table}

As shown in Table \ref{tab2}, the computational complexity of the proposed TT-rBKI method and the TT-rSI method is $\mathcal{O}(I^{N}s^2q^2 +\sum_{i=1}^{N-1}rI^{N-i}s^2q^2)$ and $\mathcal{O}(I^Nsq +\sum_{i=1}^{N-1}rI^{N-i}sq)$, respectively. %
Based on the work in \cite{Huber2017} and \cite{Q2022}, we know that the computational complexity of TT-rSVD is $\mathcal{O}(NsI^{N} )$, which is significantly reduced compared to TT-SVD. Among these three randomized algorithms (i.e., TT-rSVD, TT-rSI and TT-rBKI), TT-rBKI theoretically has slightly higher computational complexity. The extensive numerical experiments in Section \ref{Sec:results} below will demonstrate that the proposed TT-rBKI can achieve the best quality for TT approximation.

\begin{algorithm}
		\renewcommand{\algorithmicrequire}{\textbf{Input:}}
		\renewcommand{\algorithmicensure}{\textbf{Output:}}
		\caption{ Randomized TT approximation with subspace power iteration (TT-rSI)}
		\label{alg:tt-rsi}
		\begin{algorithmic}[1]
			\Require Tensor $\mathcal{A} \in\mathbb{R}^{I_1\times I_2\times\cdots\times I_N}$, iteration $q$, target rank $\textbf{r} = [r_1,r_2,\cdots ,r_{N-1}]$, oversampling parameter $p$, and $r_0 =1$.
			\Ensure $\mathcal{Q}_{1},\mathcal{Q}_{2},\cdots,\mathcal{Q}_{N}$.
			
			\For{$n=1,2,\cdots,N-1$}
			\State $A = {\tt reshape}(\mathcal{A},[r_{n-1}I_n,I_{n+1} I_{n+2}\cdots  I_{N}])$.
			
			\State Create random Gaussian matrix $\Omega \in R^{I_{n+1} I_{n+2}\cdots  I_{N} \times (r_n+p)} $.
			\State  $Y = A\Omega$.
			\State Calculate $[Q_0,\sim] = {\tt qr}(Y,0)$.
			
			\For{$j=1,2,\cdots,q$}
			\State $\hat{Y}_j$=$A^\top Q_{j-1}$.
			\State $(\hat{Q}_j ,\sim) = {\tt qr}(\hat{Y}_j,0)$.
			\State $Y_j =A\hat{Q}_j$.
			\State $(Q_j,\sim) = {\tt qr}(Y_j,0) $.
			\EndFor  
			\State $Q = Q(:,1:r_n)$.
			\State Form factor tensor $ \mathcal{Q}_{n} = {\tt reshape}(Q,[r_{n-1},I_n,r_n])$.
			\State Update $ \mathcal{A}  = \mathcal{A} \times_1 Q^\top \in R^{r_n \times I_{n+1}\times\cdots\times I_N} $.
			\EndFor  
			\State Form the last factor tensor with $r_N=1 $, $ \mathcal{Q}_{N} = {\tt reshape}(\mathcal{A},[r_{N-1},I_N,r_N])$.
			\State \textbf{return} $\mathcal{Q}_{1},\mathcal{Q}_{2},\cdots,\mathcal{Q}_{N}$.
		\end{algorithmic}
	\end{algorithm}	
 
\section{Numerical Experiments} \label{Sec:results}
In this section, we test the proposed TT-rBKI algorithm on synthetic and real-world data  with and without noise, and make comparison with the related state-of-the-art methods, i.e., TT-SVD, TT-rSVD and TT-rSI.


	%

	
The quality of the reconstructed tensor is measured by the peak signal-to-noise ratio (PSNR) say $\rho_{\rm psnr}$ and relative error say $\epsilon_{\rm err}$. For tensor $\mathcal{A}\in \mathbb{R}^{I_1\times I_2 \times I_3}$ and its low-rank TT approximation $\hat{\mathcal{A}}$, the PSNR is defined as
\begin{equation*}
		\rho_{\rm psnr} =	10\cdot\log_{10}\dfrac{I_1 I_2 I_3 \Vert\hat{\mathcal{A}}\Vert_{\infty}^2}{\|\mathcal{A}-\hat{\mathcal{A}}\|_F^2}.
\end{equation*}
The relative error of the low-rank reconstruction tensor is defined as
\begin{equation*}
		\epsilon_{\rm err} =  \parallel \mathcal{A}- \hat{\mathcal{A}} \parallel_F / \left \| \mathcal{A}  \right \| _F .
\end{equation*}
All the codes are implemented in MATLAB with the Tensor Toolbox \cite{B2023} and TT-Toolbox-master. The calculations are run on a laptop with AMD 5800H CPU (2.50GHz) and 16GB RAM.

\subsection{Experiments on noise-free data}
We first test and compare different methods on synthetic and real-world data for the noise-free case.
	\begin{figure}[htbp!]	
        \centering	
		\includegraphics[trim={{1.1in} {0.2in} {1.25in} {.0in}}, clip, width=0.96\textwidth]{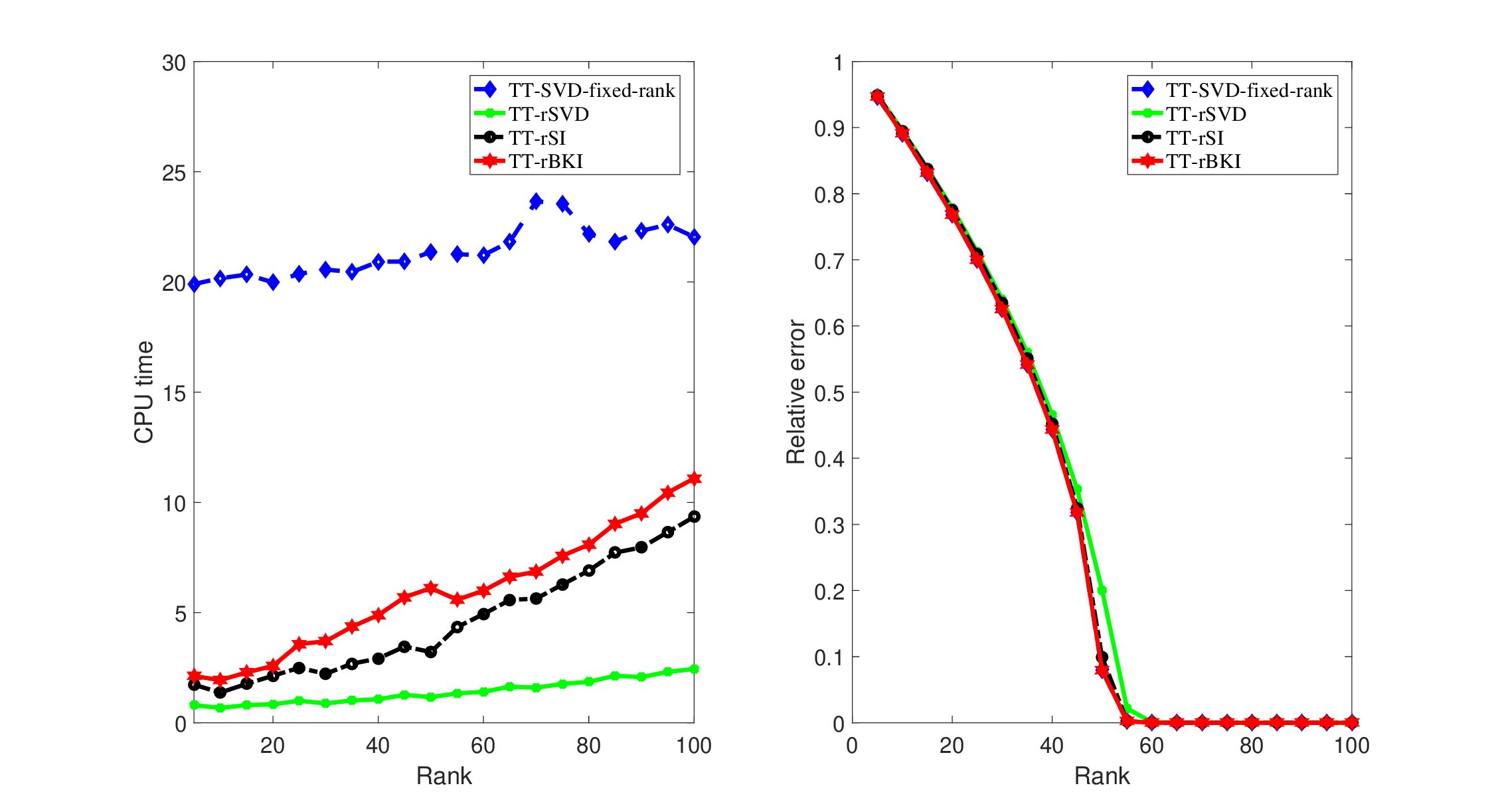}
		\caption{Results comparison on a synthetic tensor $\mathcal{A} \in \mathbb{R}^{500 \times 500 \times 500}$ with exponentially decaying spectrum in terms of CPU time (left) and relative error (right). Rank is change from 5 to 100 and the oversampling parameter is fixed at 2.}\label{fig2}	
	\end{figure}


\subsubsection{Spectrum decaying tensor}
We now test the performance of different algorithms on a synthetic tensor $\mathcal{A} \in \mathbb{R}^{500 \times 500 \times 500}$ with an exponentially decaying spectrum. The diagonal tensor with $j$-th frontal slice has the form of
	\begin{equation*}
		\mathcal{A}^{(j)} = {\rm diag}(\underbrace{1,1,\cdots,1}_{\min(T,j)},10^{-D},10^{-2D} , \cdots,10^{-(n-\min(T,j))D}).
	\end{equation*}
We fix $T =  50$ and $D = 1$. The results for this synthetic data are given in Figure \ref{fig2}. Figure \ref{fig2} shows that the relative error of TT-SVD and TT-rBKI is very similar, but TT-rBKI takes much less CPU time than TT-SVD. Moreover, TT-rSI and TT-rSVD cost less CPU time than TT-rBKI yet TT-rSVD achieves slightly worse accuracy than TT-rBKI when the rank is between 30 and 50.
	
 \begin{figure}[htbp!]
		\centering			\includegraphics[trim={{1.1in} {0.2in} {1.2in} {.0in}}, clip, width=0.96\textwidth]{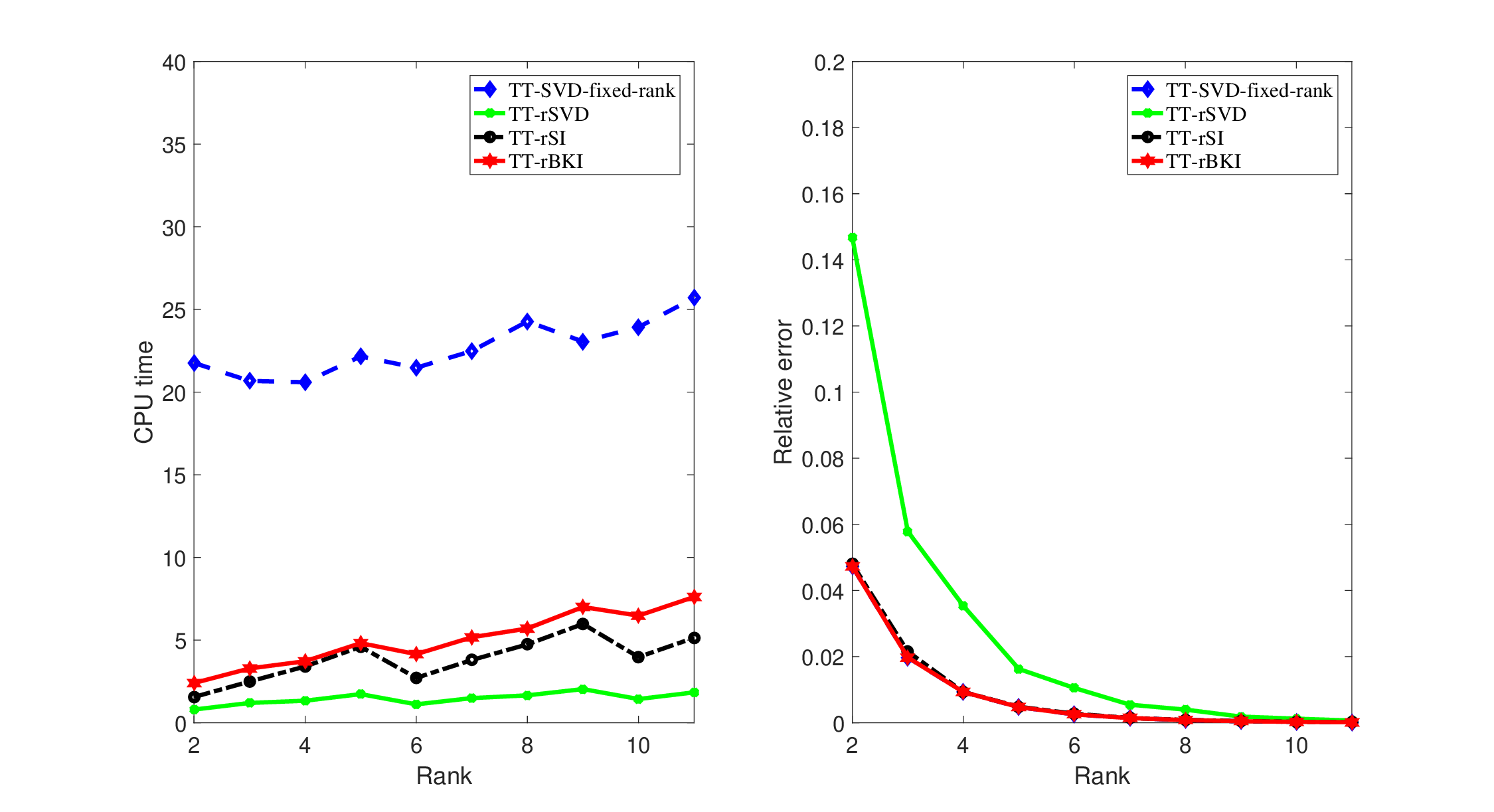}
		\caption{Results comparison on the power function data with size of $45\times45\times45\times45\times45$ in terms of the CPU time (left) and relative error (right). The target rank $[r,r,r,r]$ is changed from 2 to 11 and oversampling parameter p is fixed at 2.}\label{fig1}	
	\end{figure}
\subsubsection{Power function tensor}
A five-order power function tensor data can be generated by the following function
	\begin{equation} \label{eq-5order-functionaltensordatay}
		x_{i_1 i_2,\cdots,i_n}= \frac{1}{\sqrt[h]{i_1^h+i_2^h+\cdots+i_n^h}}, \ n=1,2,\cdots,N.
	\end{equation}
 In this synthetic data experiment, we fixed $N=5$, $h=5$, and the size of the dataset is $[45,45,45,45,45]$, while increasing the target rank $[r,r,r,r]$.  
	
	The results of the compared methods for this generated synthetic data are displayed in Figure \ref{fig1}. The right panel of the figure reveals that the relative error achieved by the TT-SVD, TT-rSI and TT-rBKI methods is very similar and much better than that of TT-rSVD. The left panel of Figure \ref{fig1} reveals that TT-rSI requires less CPU time in seconds than TT-rBKI, while achieving almost the same accuracy as TT-SVD. Moreover, Figure \ref{fig1} demonstrates that TT-rSVD is the fastest, yet its accuracy is much lower than the other methods, including our TT-rBKI.

	\subsubsection{Real-world data} \label{subsec:rwd-nf}
	In this subsection, we test different algorithms on three real-world tensors. The first real-world tensor is a gray video\footnote{http://trace.eas.asu.edu/yuv/index.html}, named hall-quif, with a size of $144\times176\times150$. The fixed TT-rank $[r_1,r_2] $ is chosen from $r_1 = r_2 = 10$ to $r_1 = r_2 = 20$.
	The second real-world tensor is derived from hyperspectral data\footnote{https://www.ehu.eus/ccwintco/index.php/Hyperspectral\_Remote\_Sensing\_Scenes} of size $1096 \times 715 \times 102$  The rank is chosen from $r_1=r_2=5$ to $r_1=r_2=50$.
	The third real-world tensor is a color video\footnote{https://github.com/a494626340/data}, i.e., a movie clip, with the size of $480 \times 848 \times 3 \times 147$. We permute the third dimension to the fourth dimension so that the number of frames can be transferred to the third dimension. The rank $[r_1,r_2,r_3]$ is chosen from $r_1 = r_2 = 5$ to $r_1 = r_2 = 50$ with $r_3 = 3$.
	The oversampling parameter was fixed at 5 for all three datasets.
	
	\begin{figure}[htbp!]
		\centering
		\includegraphics[trim={{1.1in} {0.2in} {1.2in} {.0in}}, clip, width=0.96\textwidth]{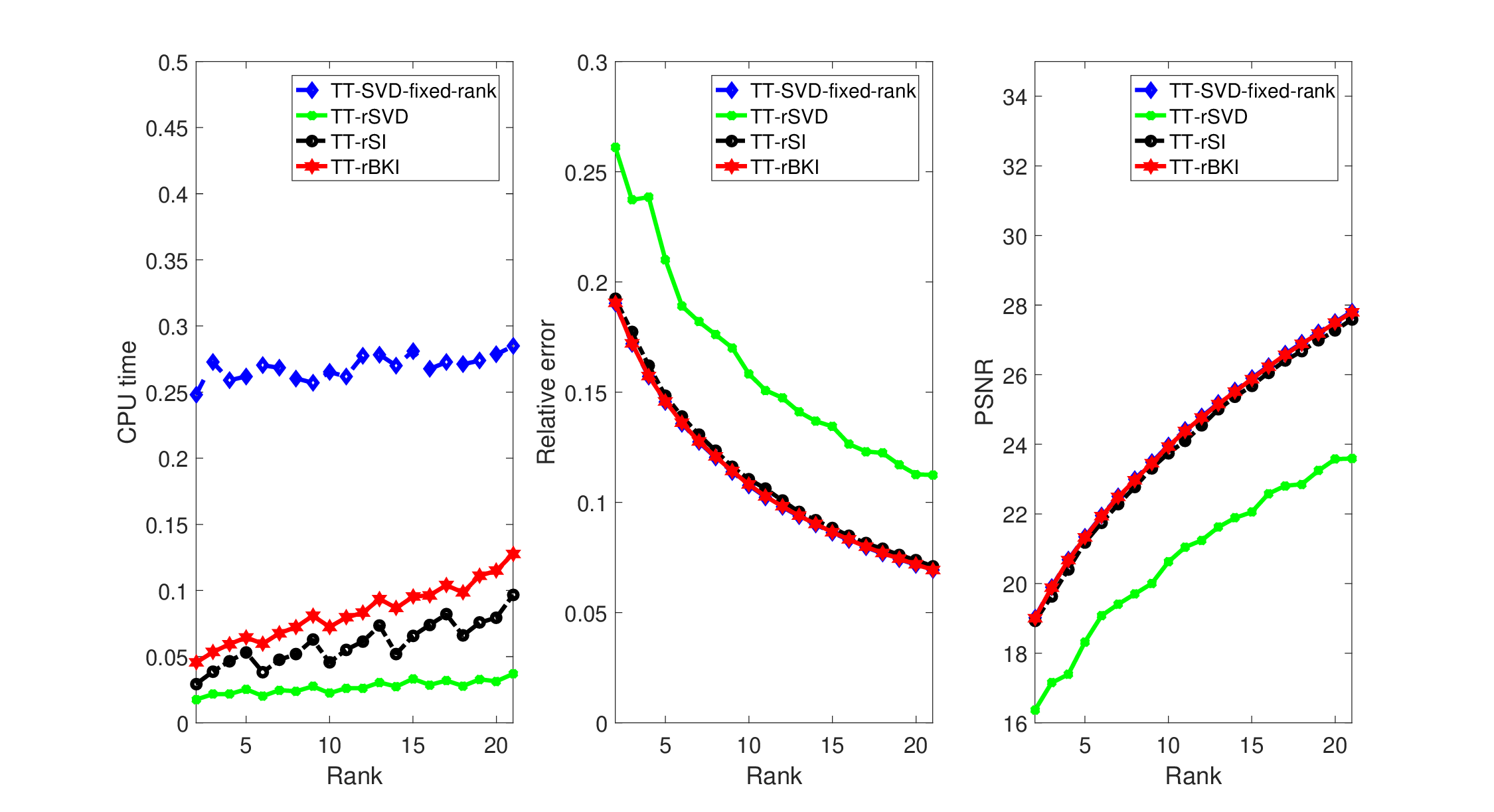}
		\caption{Results comparison on the gray video clip with size of $144\times176\times150$ in terms of the CPU time (left), relative error (middle) and PSNR (right). }
		\label{fig3}		
	\end{figure}

	\begin{figure}[htbp!]
		\centering
        \begin{tabular}{cc}
		\includegraphics[trim={{2.0in} {5.0in} {10.0in} {.5in}}, clip, width=1.4in, height = 1.2in]{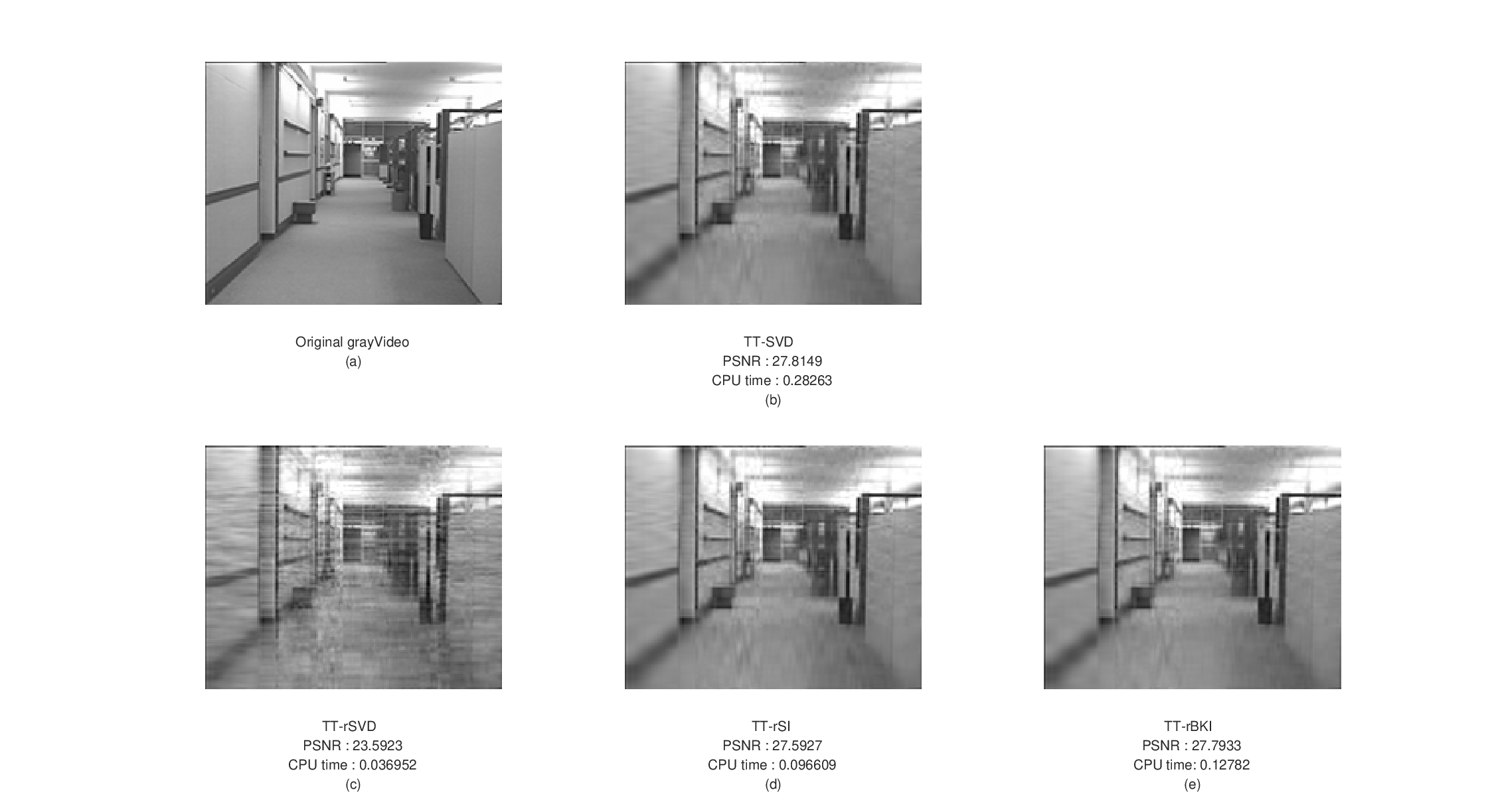} &
  	\includegraphics[trim={{6.25in} {5.0in} {5.8in} {.5in}}, clip, width=1.4in, height = 1.2in]{gray_video_pic.eps} \vspace{-0.05in} 
    \\
   {\footnotesize (a) Original gray video frame} & {\footnotesize (b) TT-SVD } \vspace{-0.05in} \\
    & {\footnotesize CPU: 0.28; PSNR: 27.8149} 
    \end{tabular}
    \begin{tabular}{ccc}
    	\includegraphics[trim={{2.0in} {1.0in} {10.0in} {4.4in}}, clip, width=1.4in, height = 1.2in]{gray_video_pic.eps} &
        \includegraphics[trim={{6.25in} {1.0in} {5.8in} {4.4in}}, clip, width=1.4in, height = 1.2in]{gray_video_pic.eps} &
        \includegraphics[trim={{10.4in} {1.0in} {1.5in} {4.4in}}, clip, width=1.4in, height = 1.2in]{gray_video_pic.eps} \vspace{-0.1in} \\
     {\footnotesize (c) TT-rSVD} & {\footnotesize (d) TT-rSI} & {\footnotesize (e) TT-rBKI} \vspace{-0.05in} \\
   {\footnotesize CPU: 0.04; PSNR: 23.5923} & {\footnotesize CPU: 0.10; PSNR: 27.5927} & {\footnotesize CPU: 0.13; PSNR: 27.7933}        
     \end{tabular}
		\caption{Low-rank approximation of one frame of the gray video clip by different methods.}
		\label{fig4}	
	\end{figure}

	\begin{figure}[htbp!]	
		\centering		\includegraphics[trim={{1.1in} {0.2in} {1.2in} {.0in}}, clip, width=0.96\textwidth]{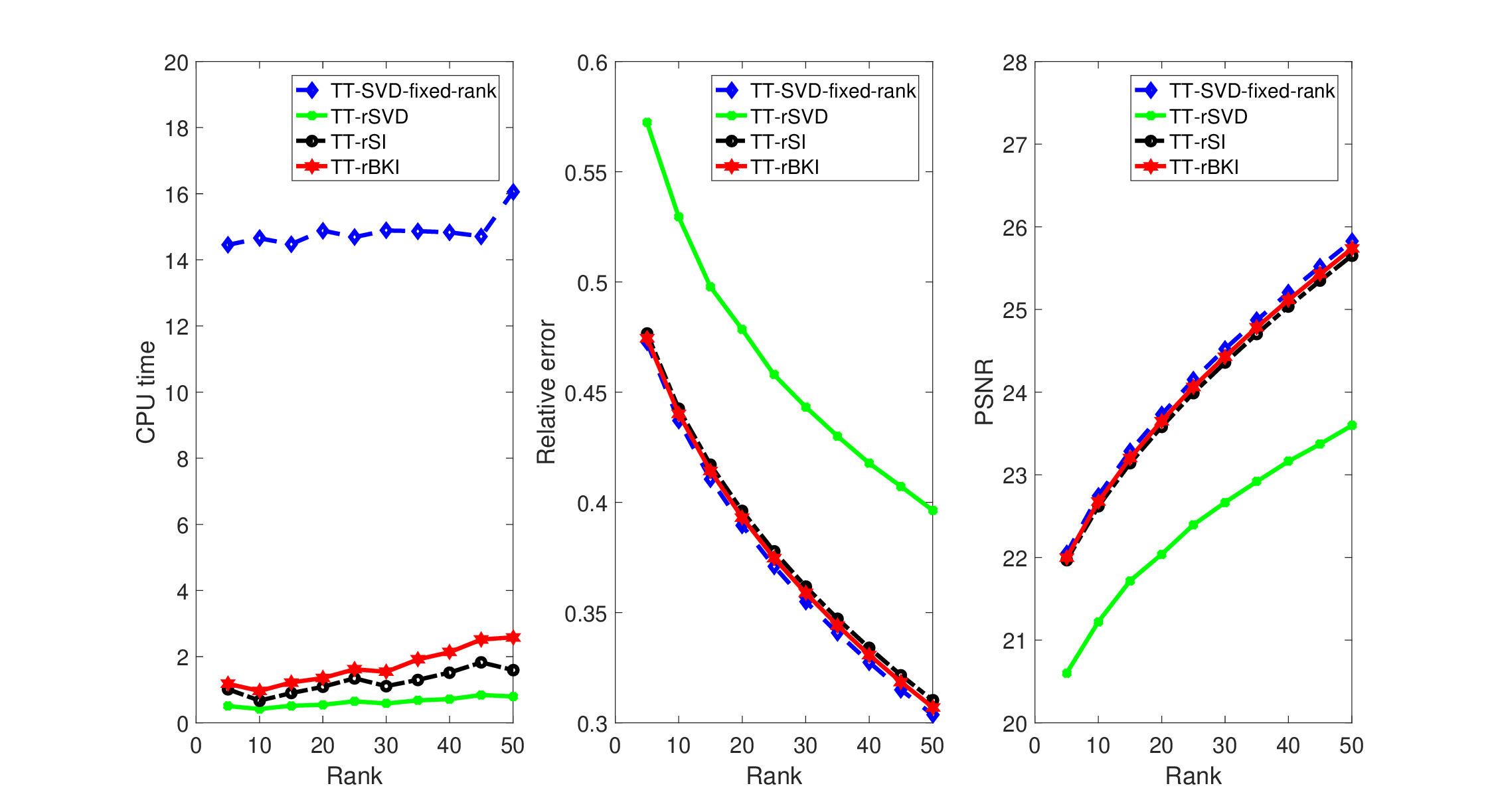}
		\caption{Results comparison on the hyperspectral tensor with size of $1096\times715\times102$ in terms of the CPU time (left), relative error (middle) and PSNR (right). }
		\label{fig5}
	\end{figure}

	\begin{figure}[htbp!]
		\centering
		\includegraphics[trim={{1.1in} {0.2in} {1.2in} {.0in}}, clip, width=0.96 \textwidth]{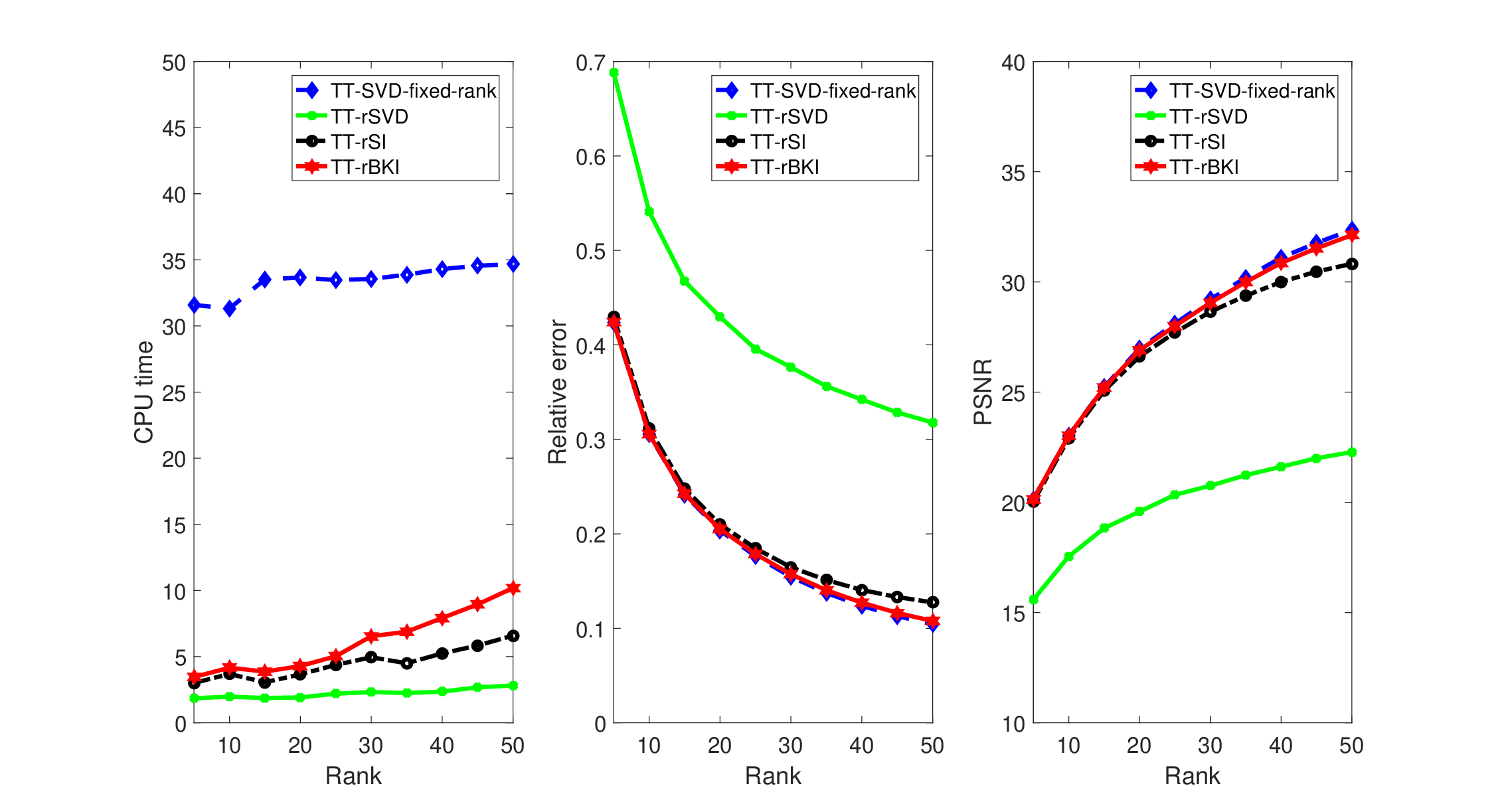}	
		\caption{Results comparison on the color video clip with size of $480\times848\times3\times147$ in terms of  CPU time (left), relative error (middle) and PSNR (right).}
		\label{fig:color-video-error}
	\end{figure}

	\begin{figure}[htbp!]	
		\centering
      \begin{tabular}{cc}
		\includegraphics[trim={{2.0in} {5.2in} {10.0in} {1.1in}}, clip, width=1.4in, height = 1.0in]{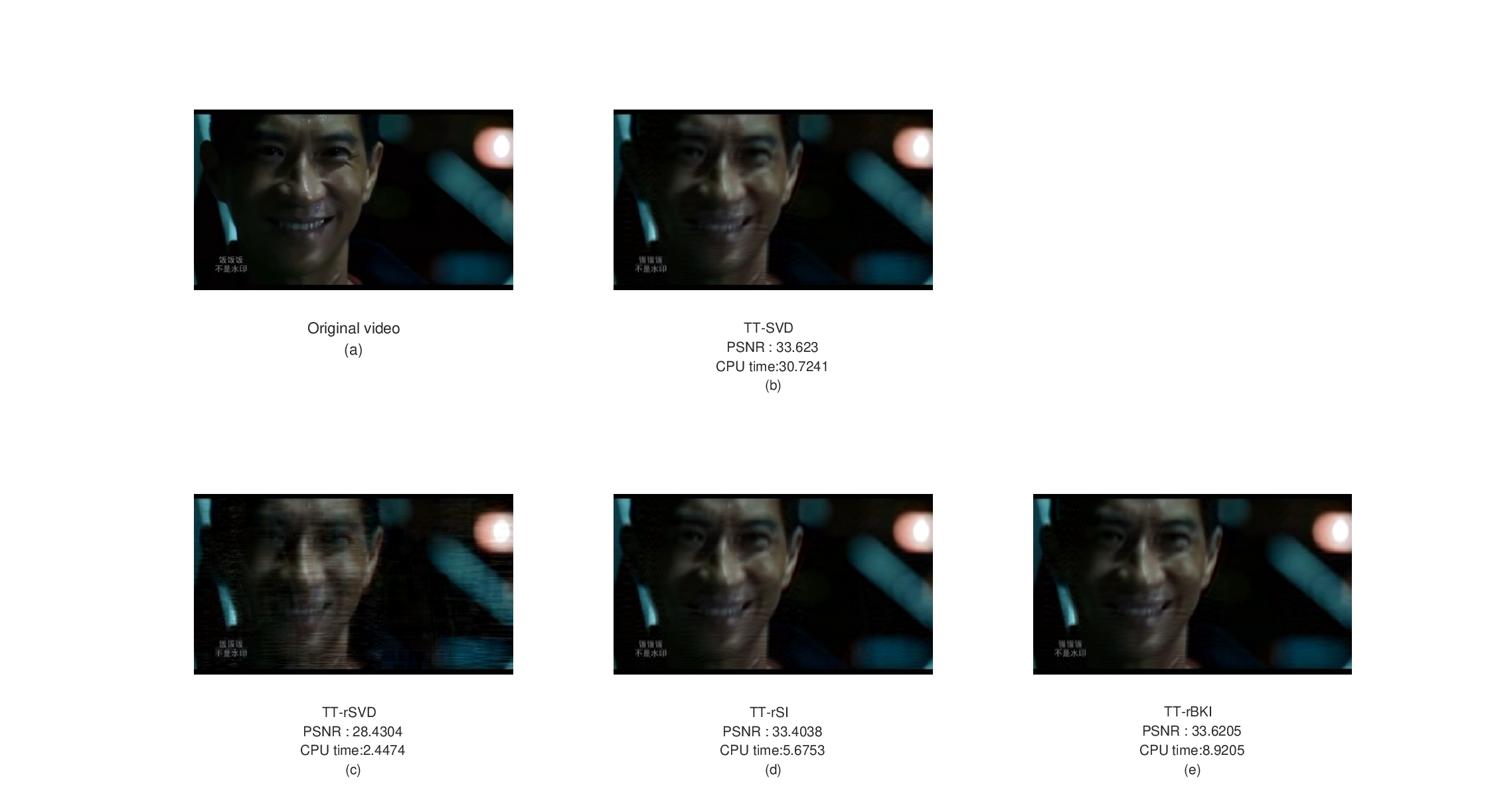} &
  	\includegraphics[trim={{6.25in} {5.2in} {5.8in} {1.1in}}, clip, width=1.4in, height = 1.0in]{video_5_pic.eps} \vspace{-0.05in} 
    \\
   {\footnotesize (a) Original color video frame} & {\footnotesize (b) TT-SVD } \vspace{-0.05in} \\
    & {\footnotesize CPU: 30.72; PSNR: 33.6230} 
    \end{tabular}
    \begin{tabular}{ccc}
    	\includegraphics[trim={{2.0in} {1.25in} {10.0in} {4.9in}}, clip, width=1.4in, height = 1.0in]{video_5_pic.eps} &
        \includegraphics[trim={{6.25in} {1.25in} {5.8in} {4.9in}}, clip, width=1.4in, height = 1.0in]{video_5_pic.eps} &
        \includegraphics[trim={{10.4in} {1.25in} {1.5in} {4.9in}}, clip, width=1.4in, height = 1.0in]{video_5_pic.eps} \vspace{-0.1in} \\
     {\footnotesize (c) TT-rSVD} & {\footnotesize (d) TT-rSI} & {\footnotesize (e) TT-rBKI} \vspace{-0.05in} \\
   {\footnotesize CPU: 2.45; PSNR: 28.4304} & {\footnotesize CPU: 5.68; PSNR: 33.4038} & {\footnotesize CPU: 8.92; PSNR: 33.6205} 
   \end{tabular}
		\caption{Low-rank approximation of one frame of the color video clip by different methods.}
		\label{fig:color-video-visual}
	\end{figure}

The comparison between different methods is again conducted in terms of the relative error, PSNR and CPU time. The experimental results on the gray video clip are shown in Figures \ref{fig3} and \ref{fig4}. The experimental results on the hyperspectral tensor are shown in Figure \ref{fig5}. The experimental results on the color video clip are shown in Figures \ref{fig:color-video-error} and \ref{fig:color-video-visual}.
	
As we can see from Figures \ref{fig3}--\ref{fig:color-video-visual}, TT-SVD takes more time on progressively larger data sets (i.e., tensor with larger size), while the time required by the randomized algorithms (i.e., TT-rSVD, TT-rSI and TT-rBKI) is much less and does not increase significantly. TT-rSVD, although it requires the shortest time, is not sufficiently accurate.  TT-rBKI takes slightly more time, but it achieves results that are closer to those of TT-SVD and are better than those of TT-rSVD and TT-rSI in terms of relative error. On the whole, the proposed TT-rBKI method benefits from the superiority of the randomized method and subspace iterations in obtaining nearly the same accuracy as that of TT-SVD by taking far less time.

\subsection{Experiments on noisy data}
We now test and compare different methods on real-world and synthetic data with Gaussian white noise. The built-in MATLAB function ${\tt awgn}$ is used to add noise.
The difference from noise-free data is that the noisy data are usually heavy-tailed.

\begin{figure}[htbp!]			
		\centering
\includegraphics[trim={{1.1in} {0.2in} {1.2in} {.0in}}, clip, width=0.96\textwidth]{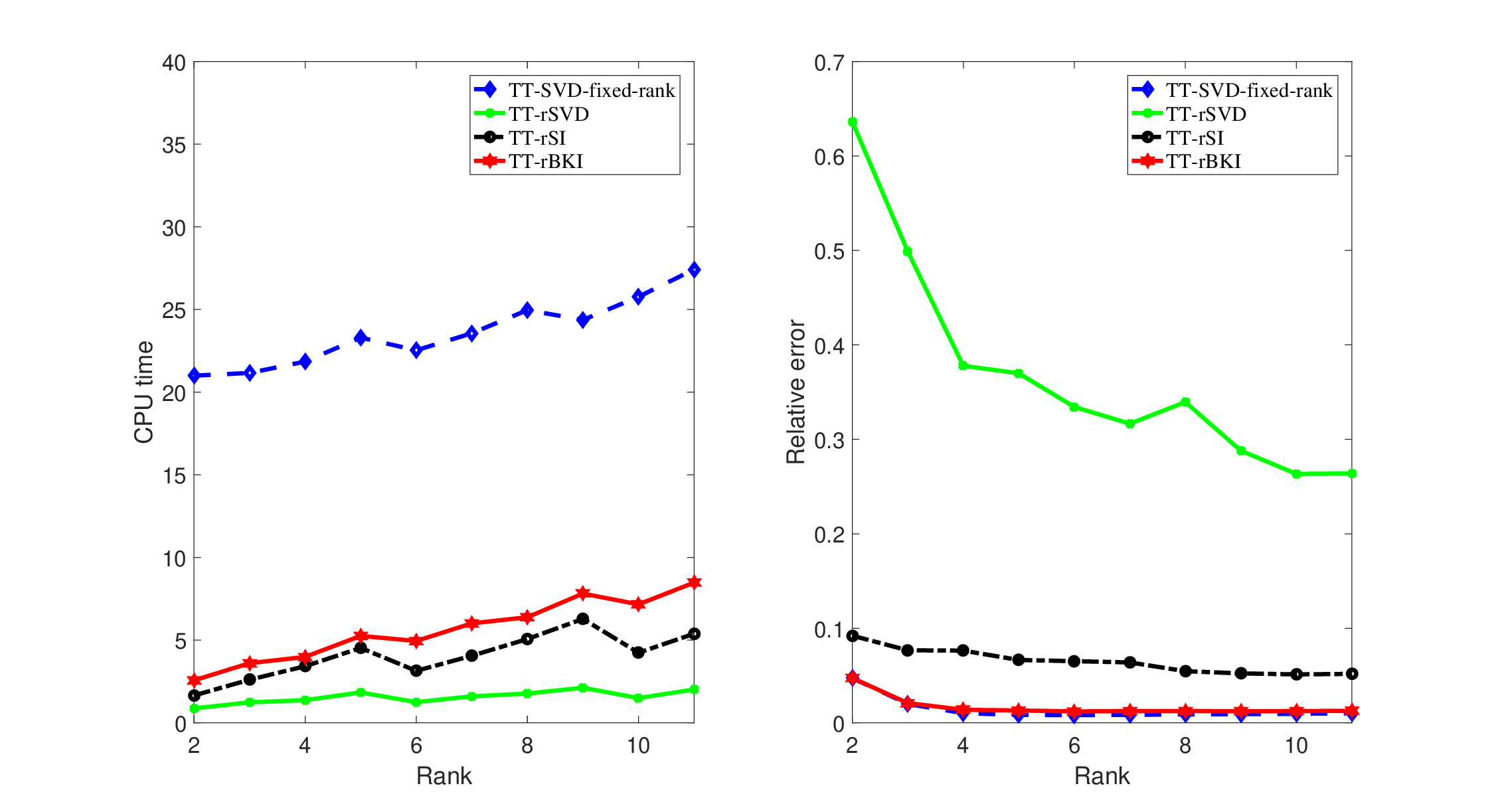}
		\caption{Results comparison on the noisy power function data with size of $45\times45\times45\times45\times45$ in terms of the CPU time (left) and relative error (right). The target rank $[r,r,r,r]$ is changed from 2 to 11 and the noise level is fixed at 5dB.}
		\label{fig8}					
	\end{figure}

	\begin{figure}[htbp!]	
		\centering	\includegraphics[trim={{1.1in} {0.2in} {1.2in} {.0in}}, clip, width=0.96\textwidth]{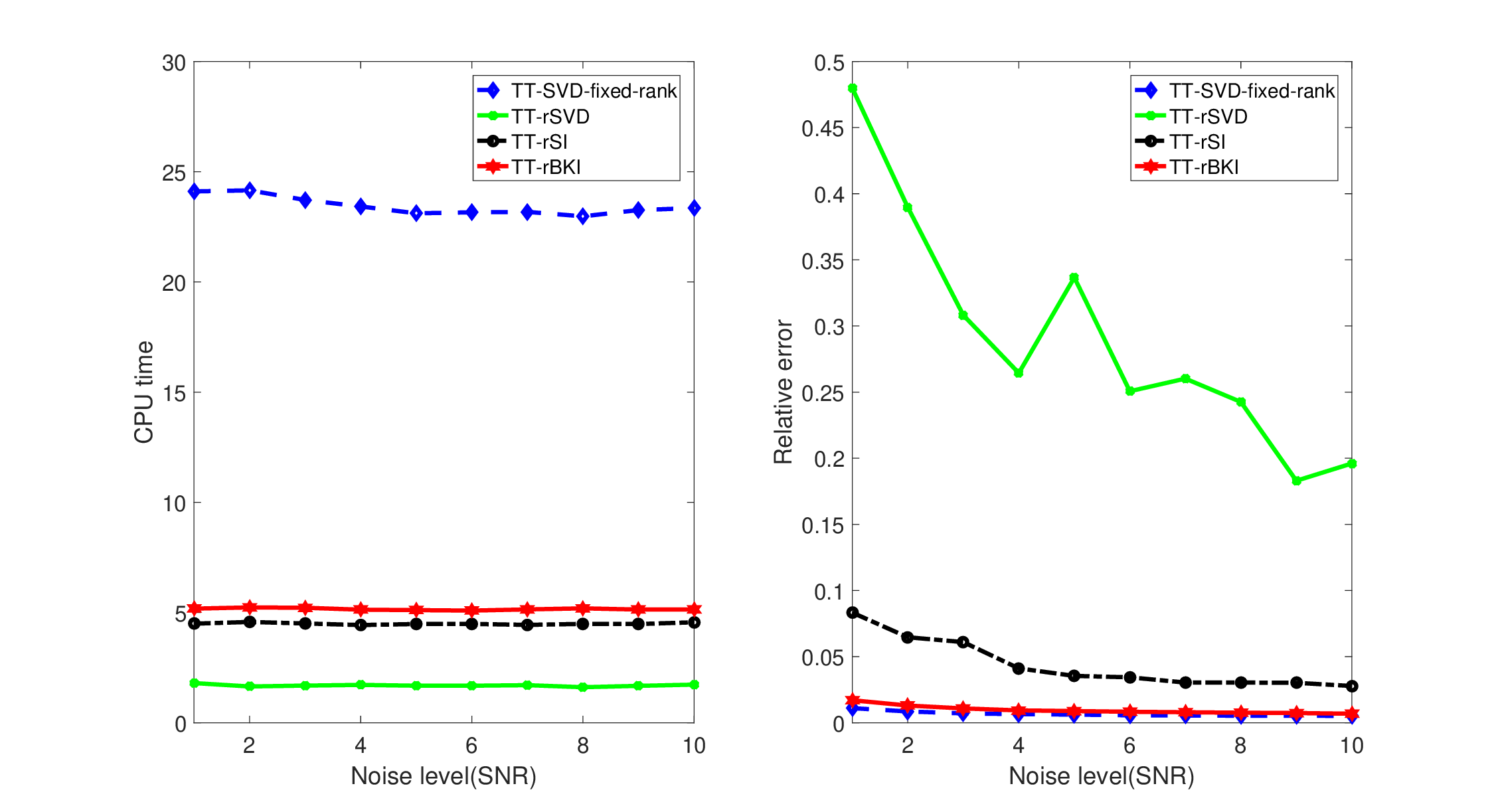}
		\caption{Results comparison on the noisy power function data with size of $45\times45\times45\times45\times45$ and decreased noise level in terms of the CPU time (left) and relative error (right). The target rank $[r,r,r,r]$ is fixed at $[5,5,5,5]$ and the oversampling parameter is fixed at 2.}
		\label{fig9}	
	\end{figure}

\subsubsection{Synthetic data}
We first conduct test to evaluate the impact of noise on the performance of randomized algorithms using synthetic data. The data of a five-order power function defined in Eq. (\ref{eq-5order-functionaltensordatay}) is generated with $N = 5$ and $h = 5$.  In this synthetic data experiment, we fix the noise level (i.e., SNR) at 5dB and keep increasing the target rank $[r,r,r,r]$. To illustrate the effect of high noise on individual methods, we fix the target rank and continuously attenuate the noise level.
The results are presented in Figures \ref{fig8} and \ref{fig9}. It shows that in the presence of noise, the proposed TT-rBKI achieves the best reconstruction quality among the randomized methods and requires much less time compared to TT-SVD. In comparison, TT-rSVD and TT-rSI achieve poor approximation results; and the higher the noise level, the worse the approximation.

	\subsubsection{Real-world data}
	Finally, we test all the methods on the real-world data with noise. We choose the same video clips used in Section \ref{subsec:rwd-nf}, i.e., the hall-quif gray video clip and the color movie clip, with the noise level fixed at 5dB and 2dB, respectively. The oversampling parameter was fixed at 5.

	\begin{figure}[htbp!]
		\centering			       \begin{tabular}{ccc}
		\includegraphics[trim={{2.0in} {5.0in} {10.0in} {.5in}}, clip, width=1.4in, height = 1.2in]{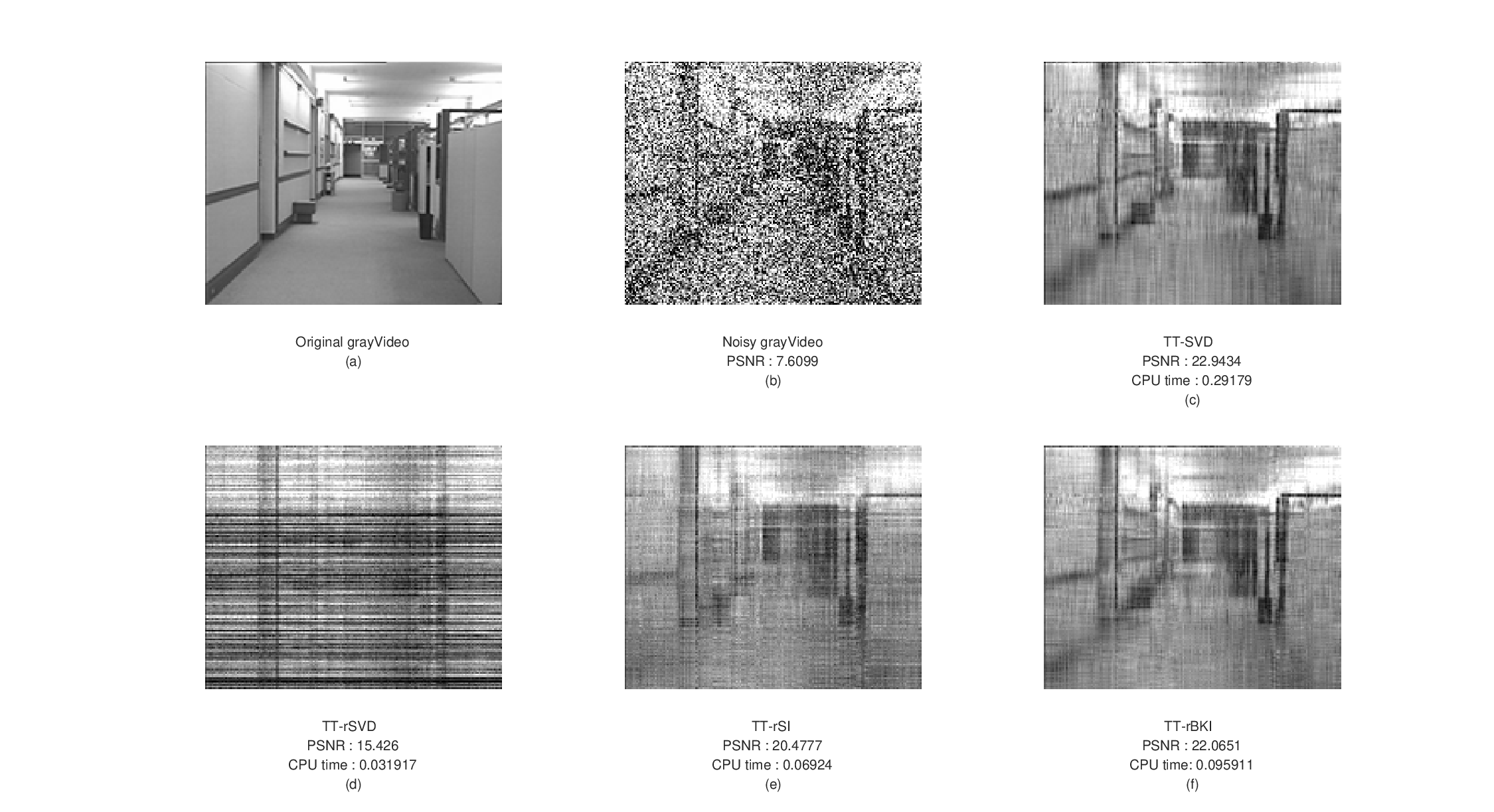} &
  	\includegraphics[trim={{6.25in} {5.0in} {5.8in} {.5in}}, clip, width=1.4in, height = 1.2in]{gray_video_SNR_2_pic.eps} \vspace{-0.05in} &
     \includegraphics[trim={{10.4in} {5.0in} {1.5in} {.5in}}, clip, width=1.4in, height = 1.2in]{gray_video_SNR_2_pic.eps} \vspace{-0.05in}
    \\
   {\footnotesize (a) Original gray video frame} & {\footnotesize (b) Noisy gray video frame} & {\footnotesize (c) TT-SVD } \vspace{-0.05in}  \\
    & {\footnotesize PSNR: 7.6099} & {\footnotesize CPU: 0.29; PSNR: 22.9434} \\
    	\includegraphics[trim={{2.0in} {1.0in} {10.0in} {4.4in}}, clip, width=1.4in, height = 1.2in]{gray_video_SNR_2_pic.eps} &
        \includegraphics[trim={{6.25in} {1.0in} {5.8in} {4.4in}}, clip, width=1.4in, height = 1.2in]{gray_video_SNR_2_pic.eps} &
        \includegraphics[trim={{10.4in} {1.0in} {1.5in} {4.4in}}, clip, width=1.4in, height = 1.2in]{gray_video_SNR_2_pic.eps} \vspace{-0.1in} \\
     {\footnotesize (d) TT-rSVD} & {\footnotesize (e) TT-rSI} & {\footnotesize (f) TT-rBKI} \vspace{-0.05in} \\
   {\footnotesize CPU: 0.03; PSNR: 15.4260} & {\footnotesize CPU: 0.07; PSNR: 20.4777} & {\footnotesize CPU: 0.10; PSNR: 22.0651}        
     \end{tabular}
		\caption{One frame of the reconstructed noisy gray video clip with size of $144\times176\times150$ by different methods. The noise level is fixed at 5dB and the rank is 13.  }
		\label{fig:gray-video-noise}	
	\end{figure}

	\begin{figure}[htbp!]			
		\centering
		\includegraphics[trim={{1.1in} {0.2in} {1.2in} {.0in}}, clip, width=0.95\textwidth]{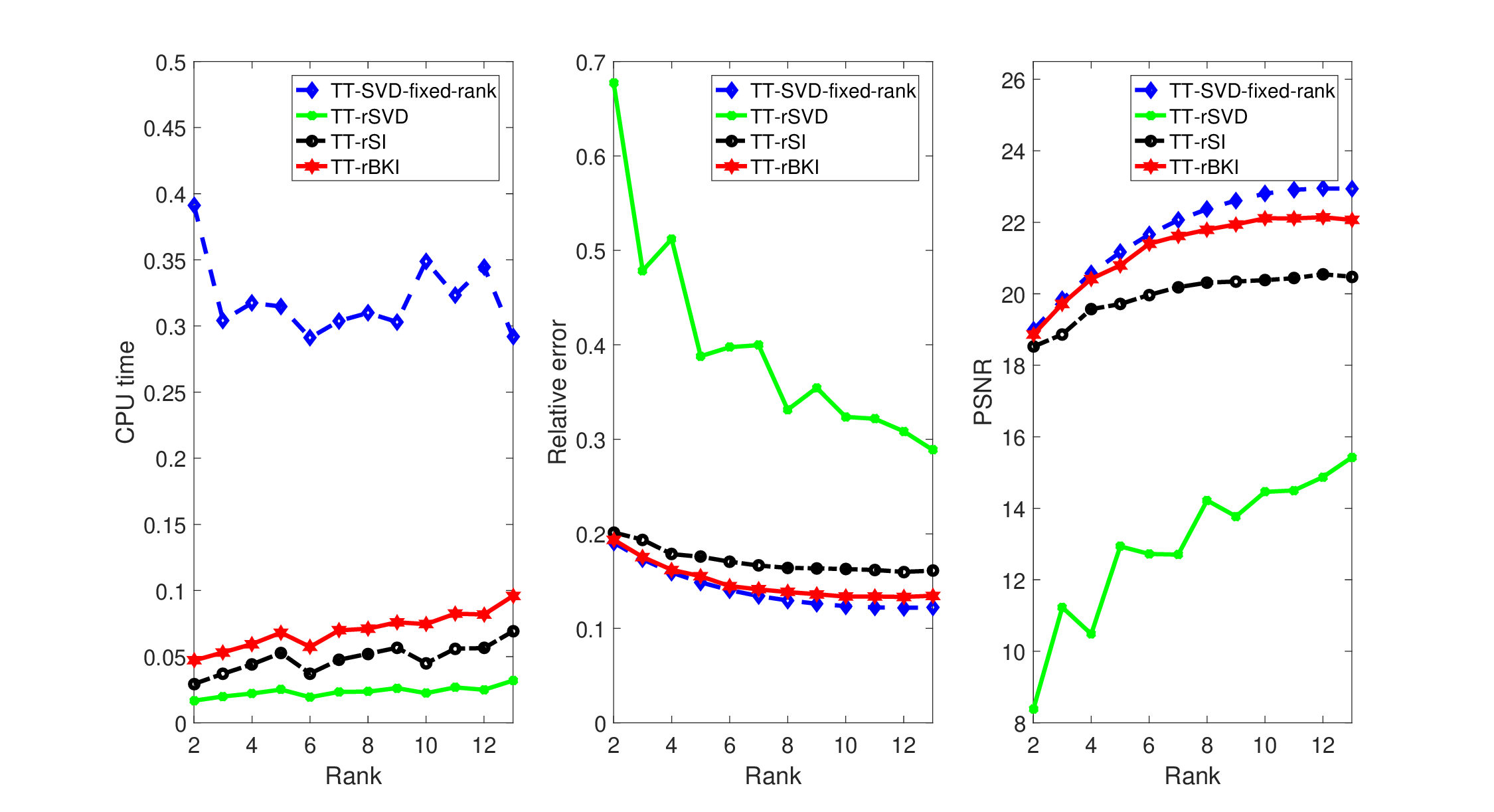}
		\caption{Results comparison on the noisy gray video clip with the size of $144\times176\times150$ by different methods in terms of the CPU time (left), relative error (middle) and PSNR (right). The noise level is fixed at 5dB. }	
		\label{fig11}
	\end{figure}

	\begin{figure}[htbp!]	
		\centering
		\includegraphics[trim={{1.1in} {0.2in} {1.2in} {.0in}}, clip, width=0.95\textwidth]{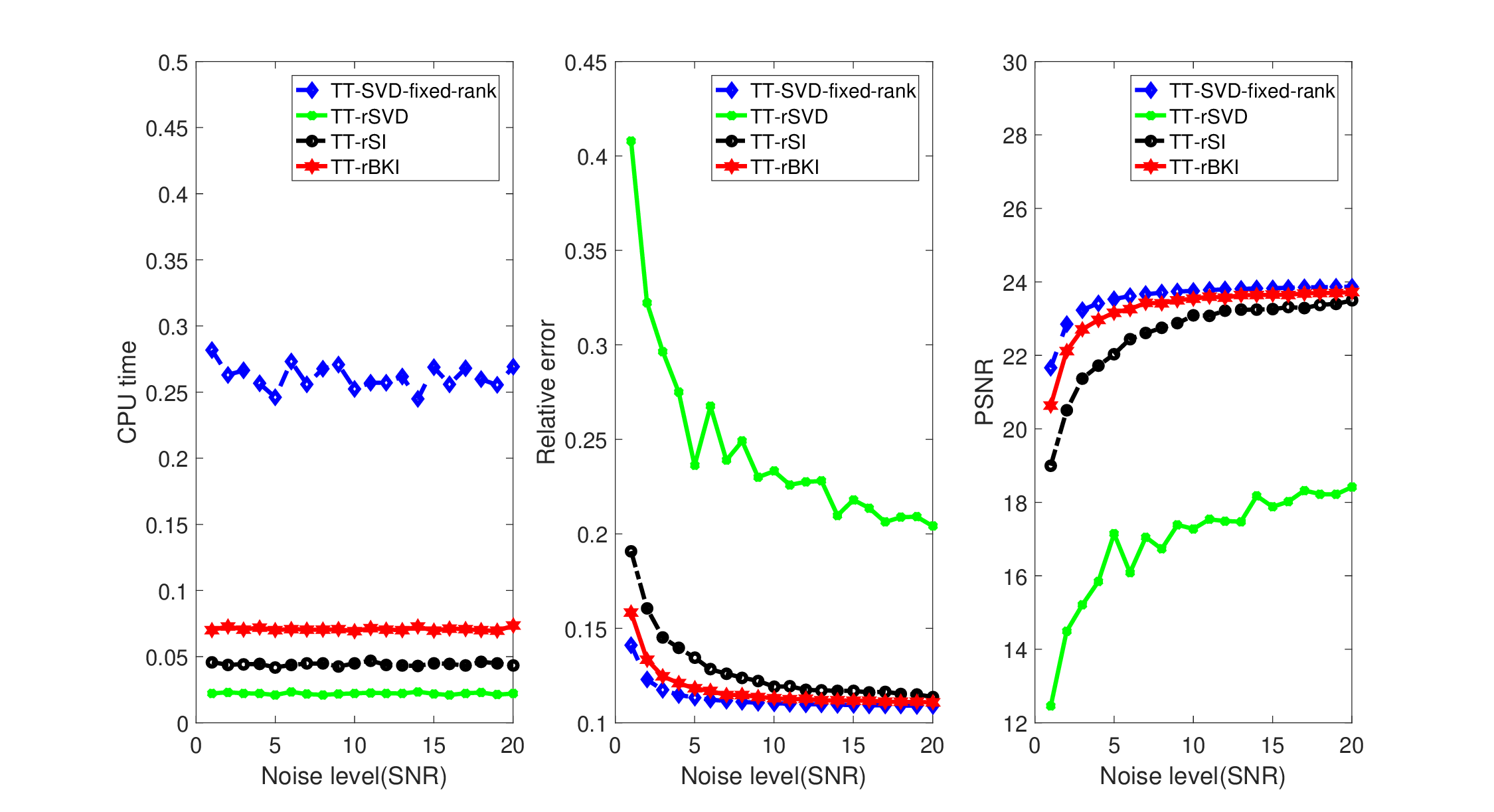}
		\caption{Results comparison on the noisy gray video clip with size of $144\times176\times150$ and decreased noise level by different methods in terms of the CPU time (left), relative error (middle) and PSNR (right). The noise level is changed from 1 to 20 and the rank is fixed at 10.}
		\label{fig12}	
	\end{figure}

	\begin{figure}[htbp!]	
		\centering			      
  \begin{tabular}{ccc}
		\includegraphics[trim={{2.0in} {5.2in} {10.0in} {1.1in}}, clip, width=1.4in, height = 1.0in]{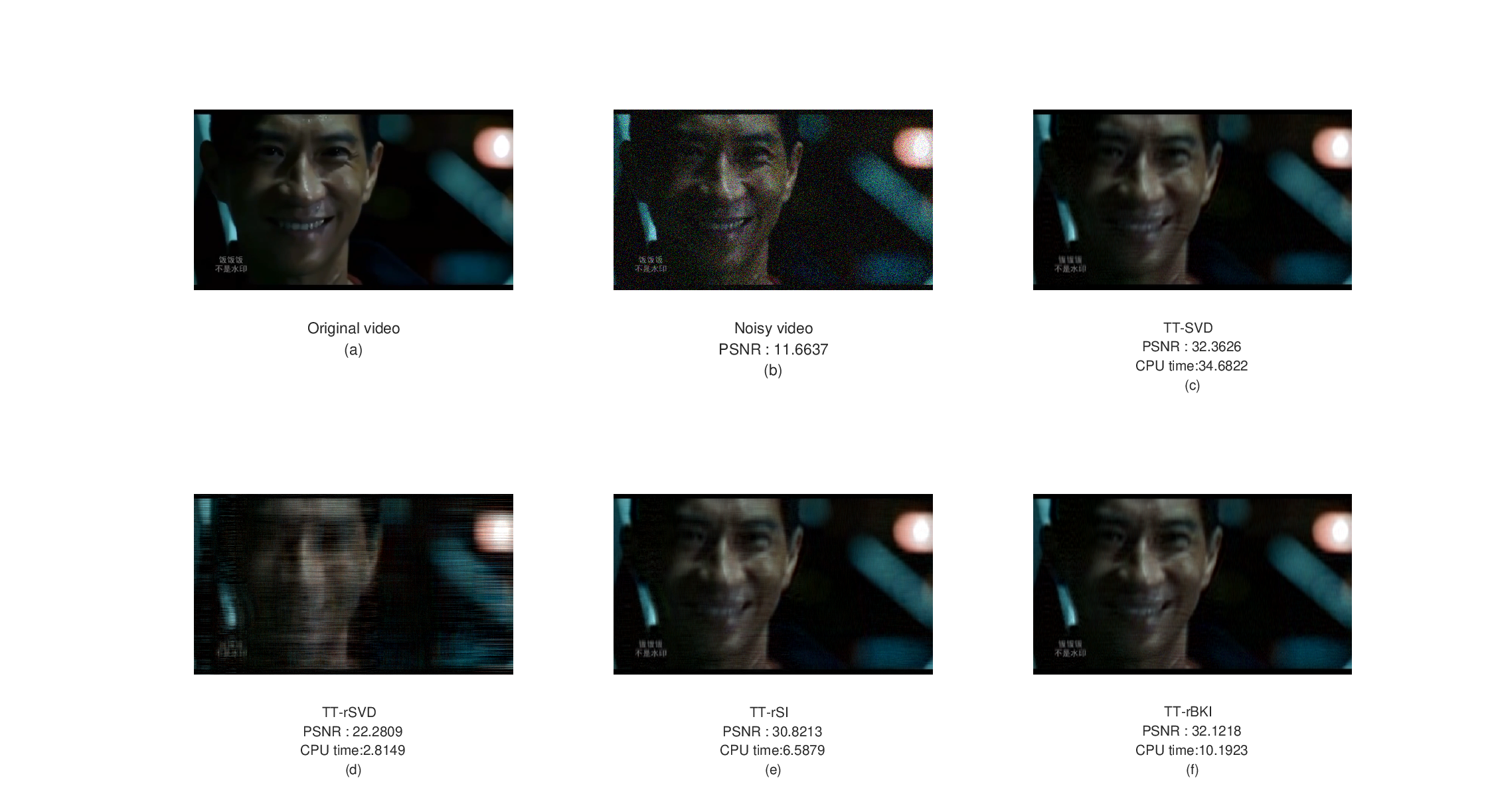} &
  	\includegraphics[trim={{6.25in} {5.2in} {5.8in} {1.1in}}, clip, width=1.4in, height = 1.0in]{Video_fixed_snr_increased_rank_pic.eps} &
   \includegraphics[trim={{10.4in} {5.2in} {1.5in} {1.1in}}, clip, width=1.4in, height = 1.0in]{Video_fixed_snr_increased_rank_pic.eps}
   \vspace{-0.05in} 
    \\
   {\footnotesize (a) Original color video frame} & {\footnotesize (b) Noisy video frame} &{\footnotesize (c) TT-SVD } \vspace{-0.05in} \\
    & {\footnotesize PSNR: 11.6637} & {\footnotesize CPU: 34.68; PSNR: 32.3626}  \\
    	\includegraphics[trim={{2.0in} {1.25in} {10.0in} {4.9in}}, clip, width=1.4in, height = 1.0in]{Video_fixed_snr_increased_rank_pic.eps} &
        \includegraphics[trim={{6.25in} {1.25in} {5.8in} {4.9in}}, clip, width=1.4in, height =1.0in]{Video_fixed_snr_increased_rank_pic.eps} &
        \includegraphics[trim={{10.4in} {1.25in} {1.5in} {4.9in}}, clip, width=1.4in, height = 1.0in]{Video_fixed_snr_increased_rank_pic.eps} \vspace{-0.1in} \\
     {\footnotesize (d) TT-rSVD} & {\footnotesize (e) TT-rSI} & {\footnotesize (f) TT-rBKI} \vspace{-0.05in} \\
   {\footnotesize CPU: 2.81; PSNR: 22.2809} & {\footnotesize CPU: 6.59; PSNR: 30.8213} & {\footnotesize CPU: 10.19; PSNR: 32.1218} 
   \end{tabular}
		\caption{One frame of the reconstructed noisy color video clip with size of $480\times848\times3\times147$ by different methods. The noise level is fixed at 2dB.  }
		\label{fig13}		
	\end{figure}

	\begin{figure}[htbp!]			
		\centering
		\includegraphics[trim={{1.1in} {0.2in} {1.2in} {.0in}}, clip, width=0.95\textwidth]{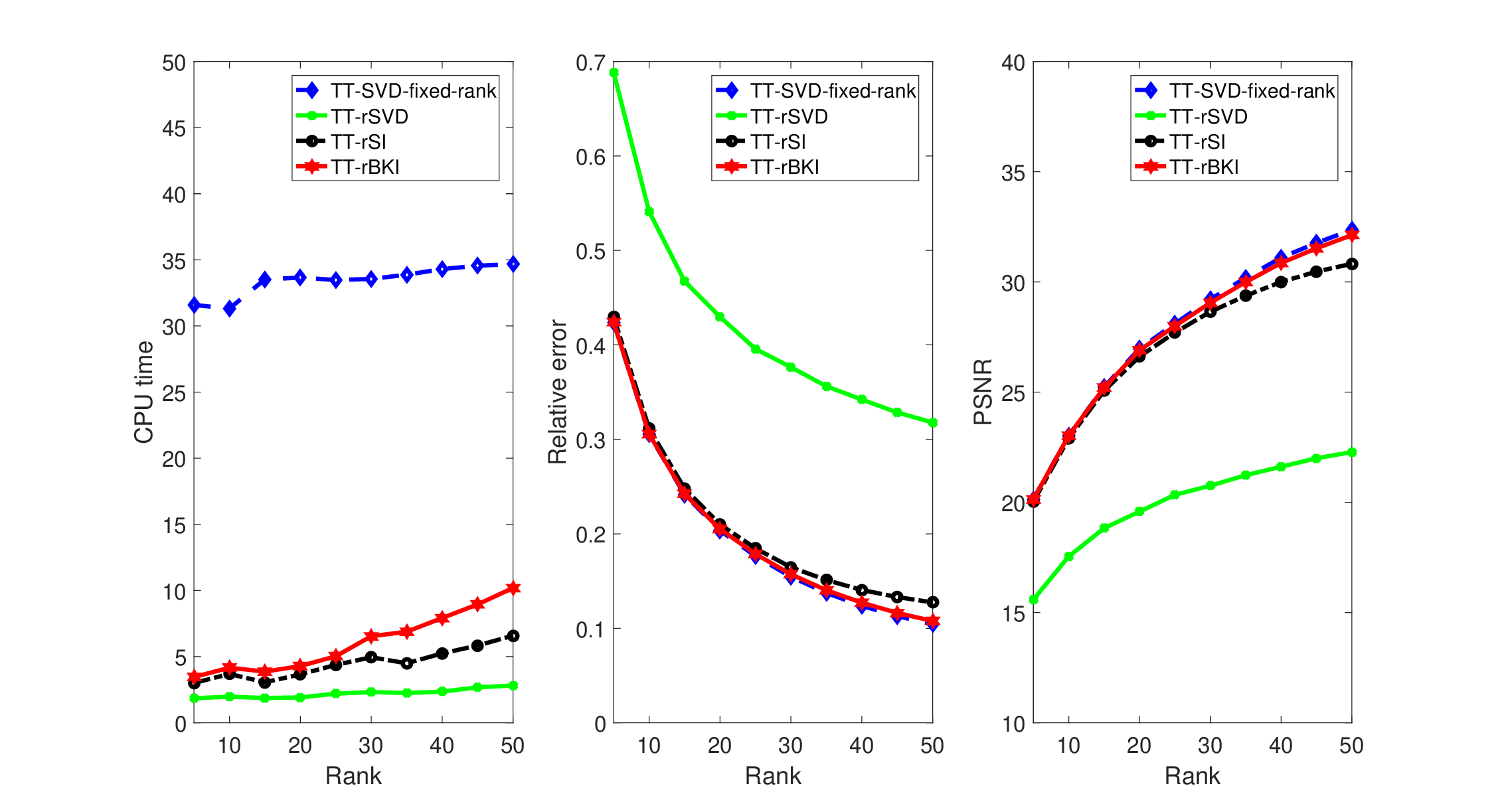}
		\caption{Results comparison on the noisy color video clip with size of $480\times848\times3\times147$ by different methods in terms of the CPU time (left), relative error (middle) and PSNR (right). The noise level is fixed at 2dB. }
		\label{fig14}		
	\end{figure}

	\begin{figure}[htbp!]	
		\centering
		\includegraphics[trim={{1.1in} {0.2in} {1.2in} {.0in}}, clip, width=0.95\textwidth]{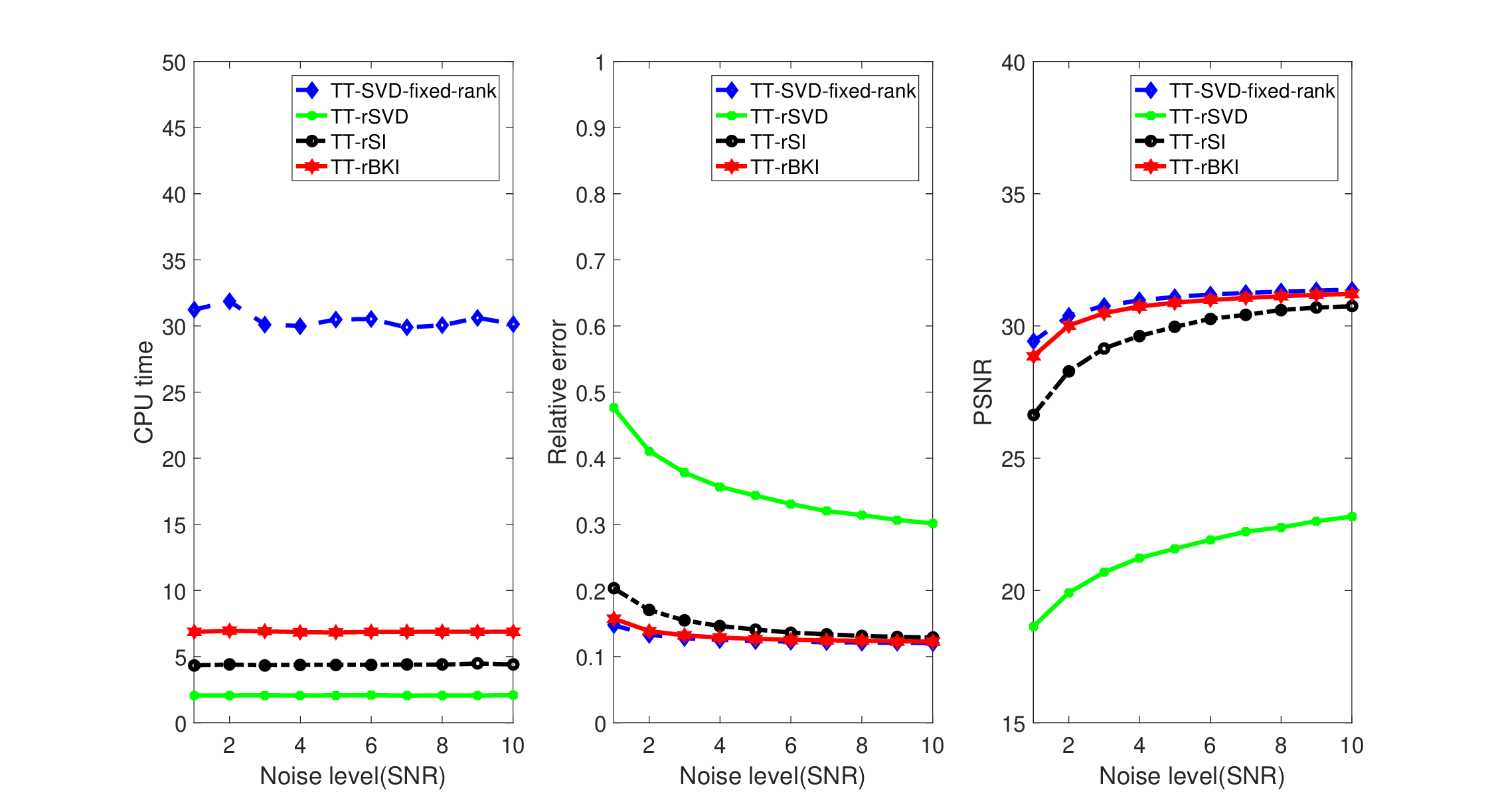}	
		\caption{Results comparison on the noisy color video clip with size of $480\times848\times3\times147$ and decreased noise level (SNR) by different methods in terms of the CPU time (left), relative error (middle) and PSNR (right). The noise level is changed from 1 to 10 and the target rank is fixed at [40,40,3]. }
		\label{fig15}		
	\end{figure}

	The experimental results are presented in Figures \ref{fig:gray-video-noise}--\ref{fig15}. In the video clips, a frame-showcase is also given in Figures \ref{fig:gray-video-noise} and \ref{fig13}; it is evident that TT-rSVD performs worse than other methods in noisy data, as the approximation it produces is unclear. All the results in Figures \ref{fig11} and \ref{fig14} also indicate that the proposed TT-rBKI is superior to TT-rSI, as it has less tail energy loss in its approximation, although it takes slightly more time. Furthermore, it is evident that the reconstruction quality of TT-rBKI, which requires much less time than TT-SVD, is closest to that of TT-SVD. Overall, the proposed TT-rBKI method performs the best among the randomized algorithms, with approximate accuracy comparable to TT-SVD but only requiring 1/5 of the CPU time used by TT-SVD. 
	Figures \ref{fig12} and \ref{fig15} give the results of the comparison of all methods at different noise levels while fixing the rank. Consistent results are obtained. In particular, the higher the noise level (i.e., the smaller the SNR), the more obvious the difference. On the whole, the proposed TT-rBKI achieves the best approximate accuracy among the randomized algorithms in all cases.

	
	\section{Conclusion} \label{Sec:con}
	In this paper, we proposed a randomized block Krylov subspace iteration, referred to as TT-rBKI, for TT approximation and validated its accuracy with a proof of its error bound. We conducted numerous experiments and comparisons on synthetic and real-world data with and without noise to demonstrate its effectiveness. The numerical results showed that the block Krylov subspace iteration can indeed significantly improve the accuracy of the randomized TT approximation. Furthermore, compared to other existing randomized algorithms, the proposed TT-rBKI is more reliable, particularly when dealing with noisy tensor data.
	
	\section*{Acknowledgements}
	This work was supported in part by National Natural Science Foundation of China (No. 12071104) and Natural Science Foundation of Zhejiang Province (No. LY22A010012, No. LD19A010002).

\section*{Ethics Declarations}
 We declare that we have no commercial or associative interest that represents a conflict of interest in connection with the work submitted.

 \section*{Data Availability}
All underlying data sets can be made available upon request or are publicly available. MATLAB codes used in this paper are available upon request to the authors. 
	

\end{document}